\documentclass[12pt,twoside]{amsart}
\usepackage{amssymb}
\usepackage{a4wide}
\usepackage[latin1]{inputenc}
\usepackage[T1]{fontenc}

\newtheorem{thm}{Theorem}[section]
\numberwithin{equation}{section}
\numberwithin{figure}{section} 
\theoremstyle{plain}
\theoremstyle{plain}    
\newtheorem{cor}[thm]{Corollary} 
\theoremstyle{plain}    
\newtheorem{prop}[thm]{Proposition} 
\newtheorem{defi}[thm]{Definition}
\newtheorem{lem}[thm]{Lemma} 
\newtheorem{rem}[thm]{Remark}
\newtheorem{ex}[thm]{Example}

\newcommand{\N}{\mathbb{N}}
\newcommand{\R}{\mathbb{R}}

\newcommand{\Cc}{\mathcal{C}}
\newcommand{\f}{\varphi}
\newcommand{\p}{\psi}
\newcommand{\vep}{\varepsilon}
\newcommand{\Ec}{\mathcal{E}}
\newcommand{\EcX}{\mathcal{E}(X,\omega)}

\newcommand{\ind}{{\bf 1}}

\newcommand{\setdef}{\ \big\vert \ }
\newcommand{\Capa}{{\rm Cap}}

\newcommand{\Capis}{{\rm Cap}_{\psi}}

\newcommand{\MA}{\mathrm{MA}\,}
\newcommand{\vol}{{\rm Vol}}
\newcommand{\ric}{{\rm Ric}}
\newcommand{\tr}{{\rm tr}}
\newcommand{\psh}{{\rm PSH}}

\newcommand{\Ric}{{\rm Ric}}

\usepackage{hyperref}
\usepackage{xcolor}
\usepackage{comment}
\usepackage{dsfont}
\usepackage{yhmath}
\definecolor{violet}{rgb}{0.0,0.2,0.7}
\definecolor{rouge}{cmyk}{0.0,0.6,0.4,0.3}
\definecolor{rouge2}{rgb}{0.8,0.0,0.2}
\hypersetup{
    bookmarks=true,         
    unicode=false,          
    pdftoolbar=true,        
    pdfmenubar=true,        
    pdffitwindow=false,     
    pdfstartview={FitH},    
    pdftitle={},    
    pdfauthor={},     
    colorlinks=true,       
   linkcolor=blue,          
    citecolor=blue,        
    filecolor=black,      
    urlcolor=cyan}           

\begin{document}
\title[the weak K\"ahler-Ricci flow]{Uniqueness and short time regularity of the weak K\"ahler-Ricci flow}
\date{\today \\
The authors are partially supported by the French ANR project MACK. The second-named author is supported by the European Research Councils}
\author[E. Di Nezza]{Eleonora Di Nezza} 

\address{Department of Mathematics, Imperial College London London, SW7 2AZ, United Kingdom}

\email{e.di-nezza@imperial.ac.uk}

\author[H.C. Lu]{Chinh H. Lu}
\address{Mathematical Sciences, Chalmers University of Technology, 
412 96 Gothenburg, Sweden}

\email{chinh@chalmers.se}

\begin{abstract}
Let $X$ be a compact K\"ahler manifold. We prove  that   the K\"ahler-Ricci flow starting from arbitrary closed positive $(1,1)$-currents  is smooth  outside some analytic subset. This regularity result is optimal meaning that the flow has  positive Lelong numbers for short time if the initial current does. We also prove that the flow is unique when starting from currents with zero Lelong numbers.
\end{abstract}

\maketitle
\tableofcontents

\section{Introduction}\label{sec:intro}
Over the last few decades, the Ricci flow introduced by Hamilton in \cite{Ham82}
$$
\frac{dg}{dt} = -2 \Ric (g)
$$
has found spectacular applications in Riemannian geometry \cite{Per1}, \cite{Per2}, \cite{Per3}. As observed by Bando, starting from a K\"ahler metric, the K\"ahler property is preserved by the Ricci flow and the resulting flow is called the K\"ahler-Ricci flow. Its long time existence is now well known thanks to the work of Cao \cite{Cao85}, Tsuji \cite{Tsu} and Tian-Zhang \cite{TZha06}. Convergence of the flow at infinity is also of great interest as shown in \cite{Cao85}, as this effect can be used to  find K\"ahler-Einstein metrics if they exist. Since then, the K\"ahler-Ricci flow, as well as its twisted version, has been studied intensively.
\medskip

Of particular interest is the attempt to run the K\"ahler-Ricci flow from singular data (see \cite{CD07}, \cite{CTZ11}, \cite{ST09}, \cite{SzTo}, \cite{GZ13}), and it can be considered as an alternative way to regularize currents. Another motivation for studying the weak K\"ahler-Ricci flow comes from an analytic analogue of Mori's Minimal Model Program. It is conjectured (see \cite{ST09}, \cite{SW13b}, \cite{SW14}) that  the minimal model of a projective algebraic variety can be obtained as the eventual limit of the K\"ahler-Ricci flow after passing through singularities. At each singular time, the flow develops singularities along some analytic subset of $X$, and so to continue through these singularities it is necessary to restart the flow from singular objects.  

In \cite{ST09}, Song-Tian succeeded in starting the flow from continuous initial data. A remarkable progress has been recently made  by Guedj and Zeriahi \cite{GZ13} which allows to define  a unique  maximal flow starting from any closed positive $(1,1)$-current. Moreover, as shown in \cite{GZ13}, the flow slowly smooths out the initial current and the speed of the regularizing effect  depends on the value of  the Lelong number of the initial current.
\medskip 

\noindent{\bf The scalar parabolic Monge-Amp\`ere equation.} 
Before going further and stating the main results of the paper, let us fix some notations.  Let $X$ be a compact K\"ahler manifold of dimension $n$ and $\alpha_0\in H^{1,1}(X,\R)$ a K\"ahler class. Fix $\eta$  a smooth closed $(1,1)$-form on $X$. The $\eta$-twisted K\"ahler-Ricci flow, introduced in \cite{CSz12}, is the following:
\begin{equation}\label{eq: KRflow}
\frac{\partial \omega_t}{\partial t} = -\ric(\omega_t) +\eta \ , \ \omega_t\vert_{t=0}= T_0 ,
\end{equation}
where $T_0$ is a fixed closed positive $(1,1)$-current in $\alpha_0$. When $T_0$ is a K\"ahler form, as mentioned above, the flow admits a unique smooth solution on a maximal interval $[0, T_{\max})$, where 
$$
T_{\max}:=\sup\{t\geq 0 \, |\, tK_X+t\{\eta\} +\alpha_0 \;\;\mbox{is nef} \}.
$$

It is standard (and more convenient) to rewrite the flow as a scalar parabolic Monge-Amp\`ere equation.  Fix $\omega$ a  K\"ahler form in $\alpha_0$. Let $\f_0$ be a global potential of $T_0$, i.e. $T_0:= \omega +dd^c \f_0$. Set 
$$
\chi:=\eta-\ric(\omega),\ \theta_t:=\omega +t\chi,
$$ 
and consider the following equation
\begin{equation}\label{eq: parabolic}
\frac{\partial \f_{t}}{\partial t} = \log \left[\frac{(\theta_t+dd^c \f_{t})^n}{\omega^n}\right], \ \f_t \to \f_0\ {\rm as}\ t\to 0.
\end{equation}
If $\f_t$ solves (\ref{eq: parabolic}) then a straightforward computation shows  that $\omega_t:=\theta_t +dd^c \f_t$ solves the flow (\ref{eq: KRflow}). 
Conversely, if $\omega_t$ solves the flow (\ref{eq: KRflow}) then it follows from the  $dd^c$-lemma  that we can write
$$
\omega_t= \theta_t + dd^c \varphi_t, 
$$
where  $\varphi_t$ solves the parabolic Monge-Amp\`ere equation (\ref{eq: parabolic}). 
\medskip

\noindent{\bf The maximal K\"ahler-Ricci flow.}
Fix $T_0=\omega+dd^c \f_0$ a closed positive $(1,1)$ current in the class $\alpha_0$. The integrability index of $T_0$ (or $\f_0$) is defined by
$$
c(T_0)=c(\f_0):=\sup\left\{\lambda>0\setdef e^{-2\lambda \f_0}\in L^1(X)\right\}.
$$
Assume that $1/2c(T_0) < T_{\max}$.  Let $\f_{0,j}$ be a sequence of smooth $\omega_0$-psh functions decreasing to $\f_0$. Such sequences exist thanks to \cite{Dem94} (see also \cite{BK07}). Let $\f_{t,j}$ be the unique solution of the parabolic equation (\ref{eq: parabolic}) with initial data $\f_{0,j}$. As shown by Guedj and Zeriahi in  \cite{GZ13}, as $j\to +\infty$ the sequence $\f_{t,j}$ decreases to $\f_t$ which satisfies the following: 
\begin{itemize}
\item For each $t>0$, $\f_t$ is a $\theta_t$-psh function. {\it Moreover, if $t> 1/2c(T_0)$, $\f_t$ is smooth on $X$ and solves (\ref{eq: parabolic}) in the classical sense.} 
\item  $\f_t$ converges in capacity to $\f_0$ as $t\to 0$.  
\end{itemize}

\noindent {\bf Remark.}
The assumption $1/2c(\f_0)<T_{\max}$ is necessary to insure that the maximal scalar solution $\f_t$ is well-defined. Without this condition  it can happen that the sequence $\f_{t,j}$ decreases to $-\infty$ (see example \ref{CondTmax}). 

\medskip

The flow constructed by approximation as above was also shown to be maximal meaning that it dominates any ``weak" solution of (\ref{eq: parabolic}). 
The regularity of this  flow was obtained for $t$ not too small  and nothing was known when $t<1/2c(T_0)$. Example 6.4 in \cite{GZ13} suggests that there might be no  regularity at all due to the presence of positive Lelong numbers. However, as in Demailly's regularization theorem \cite{Dem14}, one can expect that the regularizing effect happens outside some analytic subset. Our first result shows that it is indeed the case. 

\medskip

\noindent{\bf Theorem A.}
{\it  Assume  that 
$1/2c(T_0)<T_{\max}$. Then the maximal K\"ahler-Ricci flow starting from $T_0$ is smooth in a Zariski open  subset of $X$.} 
\medskip

More precisely, the Zariski open subset  in Theorem A is described by the complement of 
Lelong superlevel sets of $T_0$. For each $s>0$ we define 
$$
D_{s}:= \{x\in X\setdef \nu(T_0,x) \geq s\}. 
$$
The precise statement of Theorem A says that for each $\vep>0$ there exists an analytic subset $Y_{\vep}:=D_{\kappa(\vep)}$ such that $\omega_t$ is smooth in $X\setminus Y_{\vep}$ for every $t>\vep$. Here 
the constants $\kappa(\vep)$ depends only on $\vep$ and $T_0$ in such a way that it decreases to $0$ as $\vep$ goes to $0$. These constants $\kappa(\vep)$ are approximately $\vep$, up to an error term in Demailly's regularizing process. An interesting particular case is when $\omega$ represents the first Chern class of an ample holomorphic line bundle over $X$. In this case the error term is zero and $\kappa(\vep)$ can be taken to be $\vep$.  Siu's theorem \cite{Siu74} guarantees that $Y_{\vep}$ are analytic subsets of $X$. 

We also show that our result in Theorem A is optimal in the sense that 
any maximal K\"ahler-Ricci flow starting from currents having positive Lelong numbers has positive Lelong numbers in short time (see Theorem \ref{thm: singularity}).

Our previous analysis gives rise to the following definition. 
\medskip

\noindent {\bf Definition.} {\it  
A function $\f:(0,T_{\max}) \times X \to \R$ is called a weak solution of the equation (\ref{eq: parabolic}) starting from $\f_0$ if the following conditions are satisfied:
\begin{itemize}
\item[(i)] For each $t>0$, the function $\f_t$ is $\theta_t$-psh on $X$,
\item[(ii)] The function $t\mapsto \f_t$
is continuous as a map from $[0,T_{\max})$ to $L^1(X)$.
\item[(iii)]  For each $\vep>0$ there exists an analytic subset $D_{\vep}\subset X$ such that the function 
$$
(t,x)\mapsto \f_{t}(x)
$$
is smooth on $[\vep,T_{\max})\times (X\setminus D_{\vep})$ where the equation (\ref{eq: parabolic}) is satisfied in the classical sense. Moreover $D_{\vep}=\emptyset$ if $\vep>1/(2c(\f_0))$.
\end{itemize}
A family $(\omega_t)$ of closed positive $(1,1)$ currents is called a weak K\"ahler Ricci flow (\ref{eq: KRflow}) starting from $T_0=\omega+dd^c \f_0$ if $\omega_t= \theta_t +dd^c \f_t$ and $\f_t$ is a weak solution of the equation (\ref{eq: parabolic}).}

\medskip

The second goal of this paper is to study the stability and the uniqueness of the weak K\"ahler-Ricci flow. As mentioned above, the flow constructed in \cite{GZ13} is maximal among weak solutions of the parabolic equation (\ref{eq: parabolic}).
But it is not clear whether this maximal solution is the unique weak K\"ahler-Ricci flow.
When the inital current has zero Lelong numbers, we prove  that the uniqueness holds: 

\medskip

\noindent{\bf Theorem B.} 
{\it Assume that $c(T_0)=+\infty$. Then the following holds: 
\begin{itemize}
\item{\bf Uniqueness:} Any weak K\"ahler-Ricci flow starting from $T_0$  is maximal. In other words the flow is unique.
\item{\bf Stability:} Assume that $T_{0,j}$ is a sequence of positive closed $(1,1)$-currents converging to $T_0$ in the $L^1$ topology. Then  the corresponding K\"ahler Ricci flow $\omega_{t,j}$ converges to  $\omega_t$  in the following sense: for each $t\in(0,T_{\max})$, $\omega_{t,j}$ converges in $\Cc^{\infty}(X)$ to $\omega_t$ as $j\to +\infty$.
\end{itemize}
}
\medskip

The case when $T_0$ has positive Lelong number is more delicate as  the maximal flow has unbounded potential in short time. The usual method using the classical maximum principle and the pluripotential comparison principle, which are the main tools in studying parabolic equations, breaks down. However, we also have a partial result:
\medskip

\noindent{\bf Theorem B'.} {\it Assume that $c(T_0)$ is finite and $1/(2c(T_0))<T_{\max}$. Then any weak flow having the same singularities as the maximal flow is also maximal.}
\medskip

Here, we say that two flows $\omega_t, \Omega_t$ have the same singularity if  their potentials $u_t, v_t$ satisfy 
$$
u_t-C\leq v_t \leq u_t+C, \ \forall t
$$
for some constant $C$. 

\medskip

When the initial current $T_0$ has some local regularity properties it is natural to ask for  stronger convergence of the flow when the time goes to zero. If $T_0$ has continuous potential it was shown in \cite{ST09} that the solution  $\f_t$ converges uniformly to $\f_0$. When $T_0$ has finite energy the convergence is in energy as shown in \cite{GZ13}. In the last section we prove the following:

\medskip

\noindent{\bf Theorem C.} {\it Let $\f_0\in \psh(X,\omega)$ be such that $1/2c(\f_0)<T_{\max}$ and $\f_t$ be the maximal solution of (\ref{eq: parabolic}) starting from $\f_0$. 
\begin{itemize}
\item If $\f_0$ is continuous in an open subset $U\subset X$ then $\f_t$ converges locally uniformly to $\f_0$ in $U$. 
\item If $e^{\gamma\f_0}\in \Cc^{\infty}(X)$ for some positive constant $\gamma$ and $\f_0$ is strictly $\omega$-plurisubharmonic on  $X$, then $\f_t$ converges in $\Cc^{\infty}_{\rm loc}(\Omega)$ to $\f_0$, where $\Omega=\{\f_0>-\infty\}$. 
\end{itemize}
}

The first statement generalizes the uniform convergence of Song-Tian \cite{ST09}. The second one covers interesting particular cases of  currents having analytic singularities.

\medskip

The organization of the paper is as follows. In section \ref{sec: Preliminaries} we recall basic pluripotential theory that will be needed later on. In Section \ref{sec: Estimates} we establish various  a priori estimates. The proof of Theorem A  will be given in Section \ref{sec: Regularity} while the proof of Theorem B and Theorem B' will be given in Section 
\ref{sec: Stability}.  In Section \ref{sect: convergence at zero} we investigate fine  convergence of the flow at time $t=0$ when more regularity on the initial current is assumed, proving Theorem C (see Theorem \ref{thm: uniform convergence at zero} and Theorem \ref{thm: cvg zero exp}).  

\medskip

\medskip

\noindent{\bf Acknowledgement.} We are very grateful to Vincent Guedj and Ahmed Zeriahi for useful discussions and suggestions. We also thank Bo Berndtsson for many stimulating discussions during lunch time.

\section{Preliminaries}
\label{sec: Preliminaries}
In this section we recall classical results in pluripotential theory that will be needed in the sequel. We also review basic techniques in parabolic complex Monge-Amp\`ere equations.

\subsection{Recap on pluripotential theory}

Let $(X,\omega)$ be a compact K\"ahler manifold of complex dimension $n$. The set $\psh(X,\omega)$ contains all 
$\omega$-psh functions on $X$ ($\omega$-psh for short). A function $u$ is called $\omega$-psh on $X$ if it is upper semicontinuous and integrable on $X$, and $\omega+dd^c u\geq 0$ in the weak sense of currents. For a bounded $\omega$-psh function $u$, it follows from \cite{BT82} that the Monge-Amp\`ere operator $(\omega+dd^c u)^n$ is well-defined as a non-negative Borel measure on $X$. If $u\in \psh(X,\omega)$ is unbounded, we can still define the non-pluripolar Monge-Amp\`ere measure of $u$. It follows from the local property in the plurifine topology of the Monge-Amp\`ere operator that the following sequence
$$
\ind_{\{u>-j\}} (\omega +dd^c \max(u,-j))^n 
$$ 
is non-decreasing in $j$. Its limit as $j\to +\infty$ is called the non-pluripolar Monge-Amp\`ere measure of $u$, denoted by
$$
\MA(u)=(\omega +dd^c u)^n := \lim_{j\to+\infty} (\omega+dd^c \max(u,-j))^n. 
$$ 
As shown in \cite{GZ07}, $\MA(u)$ vanishes on pluripolar sets. Its total mass takes values in $[0,\int_X \omega^n]$. As defined in \cite{GZ07}, $u$ belongs to $\EcX$ if the total mass of its non-pluripolar Monge-Amp\`ere measure is maximal, i.e. $\int_X \MA(u)=\int_X \omega^n$. As shown in \cite{GZ07}, functions in $\EcX$ have mild singularities: if $u\in \EcX$, it has zero Lelong number at any point in $X$. This together with Skoda's theorem \cite{Sko} yield 
$$
\int_X e^{-A u} dV <+\infty ,  \ \forall A >0. 
$$ 
Observe that the integrability index $c(u)=+\infty$ if and only if $\nu(u,x)=0$ for any $x\in X$.

\subsection{The generalized Monge-Amp\`ere capacity}
In \cite{DL14a, DL14b} we introduce and study the so called generalized Monge-Amp\`ere capacity defined as
$$\Capis (E):=\sup\left\{\int_{E} \MA(u) \,|\, u\in \psh(X,\omega),\;\; \psi-1\leq u\leq \psi \right\}\qquad E\subset X$$
where $\psi$ is a possibly singular $\delta\omega$-psh function, $\delta\in(0,1)$, and $\MA(u)$ denotes the non-pluripolar Monge-Amp\`ere measure of $u$. Observe that if $\psi\equiv C$, then we recover the classical Monge-Amp\`ere capacity introduced in \cite{BT82}, \cite{Kol03}, \cite{GZ05}.
Such generalized capacities are the key ingredient when studying singularities of solutions of degenerate complex Monge-Amp\`ere equations and they enjoy several nice properties. Among those, it can be proven that when $\psi\in \Ec(X,\omega)$ then the $\psi$-capacity characterizes pluripolar sets, namely 
$$
\Capis^*(E)=0\Longleftrightarrow \ \text{E is a pluripolar subset of}\ X,
$$ 
where $\Capis^*$ is the outer $\psi$-capacity defined for any subset $E\subset X$ by
\begin{equation*}
\Capis^*(E):=\inf \left\{ \Capis(U) \setdef U\ \text{is an open subset of } X,\  E\subset U\right\}.
\end{equation*}
We refer the reader to \cite{DL14b} for more details.\\

\noindent Here below we recall a classical lemma (which goes back to \cite{Kol98}) that it is used in the proof of Theorem \ref{DL C0 estimate}.
\begin{lem}\label{lem: Kol tech}
Let $g:\R^+\rightarrow \R^+$ be a non-increasing, right continuous function such that $g(+\infty)=0$. Assume that there exists $C>0$ such that $g$ satisfies
\begin{equation}\label{ineq: Kol key}
s \,g(t+s)\leq C\, g(t)^2\qquad \quad \forall s\in [0,1], \, t>0.
\end{equation}
Let $t_0 >0$ be such that $g(t_0)\leq (2C)^{-1}$. Then we have 
$$
g(t)=0, \ \ \forall t\geq 2+t_0.
$$
\end{lem}
\noindent We refer the reader to \cite{Kol98} (or \cite{EGZ09}) for a proof of the above lemma.\subsection{The comparison principle}
One of the main tools in studying the K\"ahler-Ricci flow is the comparison principle. The following version is classical:

\begin{prop}
\label{prop: clas_comp_prin}
Let $u, v:\ [0,T_{\max})\times X\rightarrow \R$ be smooth functions such that $u_t,v_t$ are strictly $\theta_t$-psh functions and
$$
\dot{u_t}\leq \log \left[\frac{(\theta_t+dd^c u_t)^n}{\omega^n} \right]\quad\mbox{and}\quad \dot{v_t}\geq \log \left[\frac{(\theta_t+dd^c v_t)^n}{\omega^n} \right].
$$
Then $u_t - v_t \leq \sup_X (u_0-v_0)$. In particular if $\f_t,\psi_t$ are solutions of ($\ref{eq: parabolic}$) and $\f_0\geq \psi_0$ then $\f_t\geq \psi_t$ for all $t>0$.
\end{prop}
\noindent We refer to  \cite[Corollary 2.2]{GZ13} for a proof. \\

In dealing with the weak K\"ahler-Ricci flow it is often more convenient to use the following  generalized comparison principle, which allows the subsolution to have singularities:
\begin{prop}
\label{prop: gen comp prin}
Assume that $\f: [0,T_{\max})\times X\rightarrow \R$ is a smooth solution of the parabolic equation (\ref{eq: parabolic}). Assume that $\p$ is a weak subsolution of (\ref{eq: parabolic}) with initial data $\p_0\leq \f_0$ and $c(\p_0)=+\infty$. Then $\p_t\leq \f_t, \ \forall t\in [0,T_{\max})$.
\end{prop}
Here a weak subsolution with initial data $\p_0$ is by definition a smooth function $\p: (0,T_{\max})\times X \to \R$ such that 
$$
\dot{\p_t}\leq \log \left[\frac{(\theta_t+dd^c \p_t)^n}{\omega^n} \right], \ \forall t\in (0,T_{\max}),\ \ \p_t\to \p_0\ \ {\rm in}\ \ L^1(X).
$$
In Proposition \ref{prop: clas_comp_prin}, the functions $\f,\psi$ are  both smooth in $[0,T_{\max})\times X$ while in the above statement there is no regularity assumption on $\psi$. We also stress that the result still holds without the assumption that $c(\f_0)=+\infty$ (see Section \ref{sec: Regularity} and Section \ref{sec: Stability}). The proof without this condition is however more involved and we establish it later.   
\begin{proof}
Fix $\varepsilon>0$ and note that $\psi-\f$ is a smooth function on $[\varepsilon, T]\times X$, for any $T<T_{\max}$. It then follows from the maximum principle that $\psi-\f$ attains its maximum on $[\varepsilon, T]\times X$ at a point $(\varepsilon, x_\varepsilon)$, thus
$$\psi_t-\f_t\leq \sup_X (\psi_\varepsilon-\f_\varepsilon).$$
Now, the function $(\varepsilon, x)\rightarrow \f_\varepsilon(x)$ is smooth by assumption and in particular continuous. It then follows from Hartogs' lemma and from the fact that $\psi_t$ converges to $\psi_0$ in $L^1$ when $t\rightarrow 0$ that
$$\sup_X(\psi_\varepsilon-\f_\varepsilon)\rightarrow \sup_X(\psi_0-\f_0)\leq 0 \quad \mbox{as}\; \varepsilon\rightarrow 0.$$
It then follows that $\psi_t\leq \f_t$ for all $t>0$.
\end{proof}

\begin{rem}
One can also prove Proposition \ref{prop: gen comp prin} by using the following observation: any weak solution of the parabolic equation (\ref{eq: parabolic}) is upper semi-continuous in two variables $(t,x)$. The latter fact follows from Hartogs' lemma as the following lemma shows.  
\end{rem}

\begin{lem}
Let $T>0$ be a constant and let  $\f: [0,T]\to \psh(X,\omega)$ be a continuous map where  the latter space is equipped with the $L^1$ topology. Then $\f(t,x)$ as a function of two variables $(t,x)$ is upper semicontinuous on $[0,T]\times X$.  
\end{lem}
We thank Vincent Guedj for the following proof:
\begin{proof}
Let  $(\f_t)$ be a weak solution of the equation (\ref{eq: parabolic}) and assume by contradiction  that there exists a sequence $(t_j,x_j)$ converging to $(t,x)$ such that
$$
\limsup_{j\to+\infty} \f_{t_j}(x_j) > \f_t(x). 
$$
Since the function $t\mapsto \f_t$ is continuous as a map from $[0,T_{\max})$ to $L^1(X)$, it follows that $\f_{t_j}$ converges in 
$L^1(X)$ to $\f_{t}$. We can assume that $\f_t(x)>-\infty$ (the other case follows similarly). Take $\vep>0$ small enough such that 
$$
\limsup_{j\to+\infty} \f_{t_j}(x_j) > \f_t(x) +\vep. 
$$
It follows from the semicontinuity of $\f_t$ that there exists a small closed ball $B$ around $x$ such that
$$
\f_t(y) < \f_t(x)+\vep/2, \ \forall y\in B. 
$$
Then for $j$ big enough we will have $x_j\in B$. Thus 
$$
\limsup \f_{t_j}(x_j) \leq \limsup \sup_B \f_{t_j} \leq \sup_B \f_t < \f_t(x) -\vep/2, 
$$
as follows from Hartogs' lemma. This  yields a contradiction.
\end{proof}

\subsection{Equisingular approximation}
\label{subsec: equi}
Let $\varphi$ be a $\omega$-psh function on $X$. We are interested in approximating $\varphi$ by quasi-psh functions which are less singular than $\varphi$ and such that their singularities are somehow comparable to those of $\varphi$.  In this paper we will make use of Demailly's equisingular approximation. 
For each $c>0$, define
$$
D_c:= \left \{ z\in X \setdef \nu(\f,z) \geq c \right \}.
$$
It follows from a result of Y.T. Siu \cite{Siu74} that each $D_c$ is an analytic subset of $X$ (see \cite{Kil79} or \cite{Dem14} for alternative proofs). 
If the Lelong number of 
$\f$ is positive at some point on $X$ it was conjectured by Demailly and Koll\'ar in \cite{DK01} that $e^{-\alpha \f}$ is not integrable at the critical value $\alpha$. This openness conjecture was recently proved by Berndtsson \cite{Bern13} (see also \cite{Hiep}, \cite{GuZh} for alternative proof) while the two dimensional case was proven some years ago by Favre and Jonsson \cite{FJ05}. 

Demailly's equisingular approximation gives us  a sequence of quasi-psh functions $(\f_j)$ having analytic singularities converging to $\f$. Moreover, the sequence $(\f_j)$ also keeps track on the Lelong numbers of $\f$.

\begin{thm}[Demailly's equisingular approximation]
\cite{Dem14, Dem92}
\label{thm: Demailly equisingular}
Let $\f$ be a $\omega$-plurisubharmonic function on $X$. There exists  a sequence of quasi-plurisubharmonic functions $(\f_m)$ such that 
\begin{itemize}
\item[(i)] $\f\leq \f_m$ and $\f_m$ converges pointwise and in $L^1(X)$ to $\f$ as $m\to +\infty$. 
\item[(ii)] $\f_m$ has logarithmic singularities, $e^{2m \f_m}$ is smooth on $X$ and vanishes on the set of points $x$ such that $e^{-2m \f}$ is not locally integrable near $x$. 
\item[(iii)] $\omega + dd^c \f_m \geq -\vep_m\omega$, where $\vep_m>0$ converges to $0$ as $m\to +\infty$. 
\item[(iv)] $\int_X e^{2m(\f_m-\f)}dV <+\infty$.
\end{itemize}

\end{thm}

\begin{lem}\label{lem: Dem}
Fix $\f\in \psh(X,\omega)$. Then 
for each $\vep>0$ there exists $\kappa(\vep)>0$ and $\p_{\vep}\in PSH(X,\omega)\cap \Cc^{\infty}(X\setminus D_{\kappa(\vep)})$ such that 
$$
\int_X e^{2(\p_{\vep}- \f)/\vep} dV <+\infty. 
$$
Here, $\kappa(\vep)>0$ decreases to $0$ as $\vep \searrow 0$.
\end{lem}
\begin{proof}
Fix $\vep>0$ and let $c(\f)$ be the integrability index of $\f$, i.e. 
$$
c(\f):=\sup\left\{t>0\setdef e^{-2t\varphi}\in L^1(X)\right\}.
$$ 
We can assume that $c(\f)<+\infty$ (otherwise we are done). Let $(\f_m)$ be the Demailly's approximating sequence of $\f$ (see Theorem \ref{thm: Demailly equisingular}). Note that $\f_m$ is smooth outside the analytic subset $D_{1/m}=\{z\in X\setdef \nu(\f,z)\geq  \frac{1}{m}\}$ and 
$$
\omega + dd^c \f_m \geq -\vep_m\omega,
$$
and the constant $\vep_m$ decreases to $0$ as $m$ goes to $+\infty$.  
We let $m(\vep)>>1$ denote the smallest integer $m$ such that
$$
m>\frac{2}{\vep(1+\vep_m)} \ \ {\rm and}\ \ \frac{2\vep_m}{\vep(1+\vep_m)}<c(\f).
$$
H\"older's inequality shows that the function
$$
\psi_{\vep}:=\frac{\f_m}{1+\vep_m},
$$
is smooth outside $D_{\kappa(\vep)}$ and satisfies our requirements with $\kappa(\vep)=m(\vep)^{-1}$.
\end{proof}

\begin{rem}
The constants $\vep_m$ come from the gluing process of Demailly  \cite{Dem92}. When $\omega$ represents the first Chern class of an ample holomorphic line bundle over $X$, one can argue globally and avoid the gluing process. Hence $\vep_m$ can be taken to be zero in this case and we can choose $k(\vep)=\vep$. 
\end{rem}

\section{A priori estimates}
\label{sec: Estimates}
Let $(X,\omega)$ be a compact K\"ahler manifold of dimension $n$.  We want to run the flow from a singular initial data $\f_0\in \psh(X,\omega)$ having positive Lelong numbers. The strategy will be to pick $(\f_{0,j})$ a smooth sequence of strictly $\omega$-psh function decreasing to $\f_0$ as $j\rightarrow +\infty$ and to consider $(\f_{t,j})$ the smooth flow running from $\f_{0,j}$. The goal is to establish uniform estimates that will allow us to pass to the limit. Thus, in the sequel we work with $\f_{t,j}$ (that we will denote by $\f_t$ for simplicity) but to get uniform estimates we should always take into account the singular behavior of the initial data $\f_0$.

\subsection{Notations} From now on we fix $\vep_0>0$ and a function  $\p\in \psh(X,\omega)\cap \Cc^{\infty}(X\setminus D)$ for some  analytic subset $D\subset X$. Let $E_1,E_2$ denote the following quantities
 \begin{equation*}\label{eq: hyp fond}
E_1:=\int_X \exp{\left(\frac{2\p-2\f_0}{\vep_0}\right)} dV <+\infty,
 \end{equation*}
\begin{equation*}
\label{eq: integrability of varphi0}
E_2:=\int_X \exp{\left(\frac{-\p}{T}\right)}\omega^n <+\infty,
\end{equation*}
where $1/(2c(\f_0)) <T<T_{\max}$.
Observe that such a function $\psi$ exists thanks to Lemma \ref{lem: Dem} and $E_2$ is finite since $T>1/2c(\f_0)$ and by construction $\psi$ is less singular than $\f_0$. We also underline that $D$ is a Lelong number super-level set related to the function $\f_0$. Since we approximate $\f_0$ from above by a smooth sequence $(\f_{0,j})$, the corresponding 
constants $E_1^j$ are uniform in $j$ and we can pass to the limit when $j\to+\infty$. 
\medskip

In the remaining part of this section $\f_0$ will be a smooth strictly $\omega$-psh function on $X$. One should keep in mind that $\f_0$ in this a priori estimate section plays the role of the approximating sequence $\f_{0,j}$. 

We set, just for simplicity,
$$
\theta_t:=\omega +t \eta-t\ric(\omega). 
$$ 
Let $\f_t$ be the solution of the parabolic Monge-Amp\`ere equation:
\begin{equation*}
\frac{\partial \f_t}{\partial t} = \log \left[\frac{(\theta_t+dd^c \f_t)^n}{\omega^n}\right], \ \f_t\vert_{t=0}=\f_0,
\end{equation*}
on a maximal interval $[0,T_{\max})$. Recall that $T_{\max}$ is defined by
$$
T_{\max}:=\sup\{t\geq 0 \, |\, tK_X+t\{\eta\} +\alpha_0 \;\;\mbox{is nef} \}.
$$
The existence of the solution on this maximal interval follows from \cite{Cao85}, \cite{Tsu}, \cite{TZha06}. 
We fix $T$ and $S$ such that 
$$
\frac{1}{2c(T_0)}<T<S<T_{\max}. 
$$ 
Since our purpose is to prove regularity of the maximal solution of the equation (\ref{eq: parabolic}) on $(0,T]$, it costs no generality to assume that
$$
\frac{2\omega}{3}\leq \theta_t \leq 2\omega, \forall t\in [0,S].
$$ 
Thus  for each $t\in [0,S]$ we have
$$
\theta_t = \frac{t\theta_S}{S}    + \frac{S-t}{S}\omega \geq \left(1- \frac{t}{3S}\right)\omega.  
$$

\subsection{The $\Cc^0$ estimate}

We recall the following upper bound for $\f_t$ whose proof can be found, for example, in \cite{GZ13}.
\begin{lem}
\label{lem: upper bound varphi_t}
The following estimate holds
$$
\f_t \leq \sup_X \f_0 + (n \log 2)t, \ \ \forall t\in [0,T]. 
$$
\end{lem}

As we will see in the next sections, if we start the flow from a current with positive Lelong numbers there is no hope for the flow to be smooth. Thus a ``reasonable" $\Cc^0$ estimate must involve some singular quasi-plurisubharmonic function. In our recent works \cite{DL14a} and \cite{DL14b} we have developed the theory of generalized capacities to deal with these singular estimates. In particular we proved that if $\MA(\f) \leq C e^{-\phi}$ for some quasi-plurisubharmoinc function $\phi$ then for any $a>0$ small enough (i.e. $a\phi\in \psh (X, \omega/2)$) there exists a uniform constant $A>0$ such that $$\f \geq a\phi -A.$$

In Theorem \ref{thm: C0-estimate} and Lemma \ref{DL C0 estimate} below we use the same ideas to establish the following $\Cc^0$ estimate for the complex parabolic Monge-Amp\`ere equation (\ref{eq: parabolic}). 
\begin{thm}\label{thm: C0-estimate}
Fix $\vep>\vep_0$. For $t\in [\vep, T]$ the following estimate  holds
$$
\f_t \geq \left(1-\frac{t}{2T}\right)\p - C,
$$
where  $C$ is a positive constant depending on $\vep$,  our fixed parameters and an upper bound for $E_1$ and $E_2$.
\end{thm}

\begin{proof} Fix $t\in [\vep,T]$. It follows from \cite[Proposition 3.1]{GZ13} that 
\begin{equation}\label{derivative_estimate}
\dot{\f}_t \leq \frac{\f_t-\f_0}{t} +n,
\end{equation}
and so 
$$ 
\MA(\f_t)= e^{\dot{\f}_t }\leq  e^{n+ (\f_t-\f_0)/t} \omega^n.
$$
If $\f_0$ has zero lelong numbers then, as shown in \cite{GZ13}, the estimate (\ref{derivative_estimate}) is enough (and crucial) to get a $\Cc^0$ estimate for $\f_t$ using Ko{\l}odziej's theorem \cite{Kol98}.
However, for the general case, namely when $\f_0$ has positive Lelong numbers, this method breaks down since the bound on $\MA(\f_t)$ is no longer integrable. This is why we need to use the generalized capacities that we
introduced in \cite{DL14b}. 
  
Now set 
$$
\p_t =\left(1-\frac{t}{2T}\right)\p.
$$ 
Observe that
$$
\theta_t \geq  \left(1-\frac{t}{3S}\right)\omega \ \ \mbox{and}\ \ dd^c \p_t + \left(1-\frac{t}{2T}\right)\omega\geq 0.
$$
It then follows that $\p_t$ is $\delta \theta$-plurisubharmonic with 
$\delta \in (0,1)$ depending on $\vep_0, T, S$, namely
$$
\delta = \frac{6TS-3S\vep_0}{6TS-2T\vep_0} <1.
$$
Now, we want to apply Lemma \ref{DL C0 estimate} below with $A=1/t$ and $u=\f_0/t$ to get the desired estimate. In order to do that, we also need to bound the following quantity 
\begin{equation}
\label{eq: derivative estimate 2}
\int_X \exp{\left(\frac{q(\psi_t-\f_0)}{t}\right)}  \omega^n
\end{equation}
for some $q>1$. Indeed, by fixing $p>2$, $2/q$ being its conjugate  such that 
$$
\frac{pq}{\vep} <\frac{2}{\vep_0}
$$
and using H\"older's inequality we obtain
$$
\int_X e^{\frac{q(\p_t-\f_0)}{t}}\omega^n = \int_X e^{\frac{q(\p-\f_0)}{t}} e^{\frac{-q\p}{2T}}\omega^n \leq \left(\int_X e^{\frac{pq(\p-\f_0)}{t}} \omega^n\right)^{1/p} \left(\int_X e^{\frac{-\p}{T}}\omega^n\right)^{q/2}.
$$ 
The right-hand side of the above inequality is finite as follows from (\ref{eq: integrability of varphi0}) and from the construction of $\p$ in Lemma \ref{lem: Dem}:
$$
\int_X e^{\frac{pq(\p-\f_0)}{t}} \omega^n\leq \int_X e^{\frac{2(\p-\f_0)}{\vep_0}} \omega^n<+\infty.
$$
Thus we can bound the term in (\ref{eq: derivative estimate 2}) in terms of  $E_1$ and $E_2$. Therefore, our $\Cc^0$ estimate only depends on the fixed parameters $\vep, \vep_0, T, S$ and on upper bounds for $E_1, E_2$. The proof is thus complete.  
\end{proof}

\begin{lem}
\label{DL C0 estimate}
Assume that $\f\in \EcX$ and $(\omega+dd^c \f)^n\leq e^{A \f - u}\omega^n$, where $A>0$, $u$ is some measurable function such that $e^{A\p-u} \in L^q(X)$ for  $\p\in \psh(X,\delta\omega)$ with $\delta\in (0,1)$ and for some $q>1$. Then 
we have the following estimate 
$$
\f\geq \p -C, 
$$  
where $C$ is a positive constant depending only on $A,\delta,\omega,p$ (where $p$ is the conjugate exponent of $q$) and on an upper bound for $\int_X e^{q(A\p-u)} \omega^n$. 
\end{lem}

\begin{proof}
The proof follows the arguments in \cite{DL14a}. The goal is to prove that the function 
$$
H(t):=\left[\Capis (\f<\psi-t)\right]^{1/n}
$$
satisfies the inequality (\ref{ineq: Kol key}). Since $H(t) $ is right continuous and 
$H(t)$ goes to zero as $t$ goes to $+\infty$   (see \cite[Lemma 2.6]{DL14a}) 
we can conclude using Lemma \ref{lem: Kol tech} above.

Fix $s\in [0,1]$, $t>0$, then by \cite[Proposition 2.8]{DL14a} we get
$$
s^n\Capis (\f<\psi-t-s)\leq \int_{\{\f<\psi-t\}}\MA(\f).
$$
Using the assumption and H\"older's inequality we then have
\begin{eqnarray*}
\int_{\{\f<\psi-t\}}\MA(\f)&\leq & \int_{\{\f<\psi-t\}} e^{A\f-u}\, \omega^n\\
&=& \int_{\{\f<\psi-t\}} e^{A(\f-\psi)} e^{A\psi-u} \,\omega^n \\
&\leq & e^{-At} \vol(\{\f<\psi-t\})^{1/p}\left(\int_{\{\f<\psi-t\}} e^{q(A\psi-u)} \,\omega^n\right)^{1/q}.
\end{eqnarray*}
It follows from \cite{GZ05} that 
$$
\vol(\{\f<\psi-t\})^{1/p} \leq C(p,\omega,\delta,n) \left[\Capa_{(1-\delta)\omega}(\f<\psi-t)\right]^2.
$$
Thanks to \cite[Lemma 2.7]{DL14a} we also have 
$$
\Capa_{(1-\delta)\omega} \leq  \Capis.
$$
Putting all together yields
$$
sH(t+s)\leq C^{1/n} e^{-At/n} H(t)^2\leq B H(t)^2, 
$$
where $B=C^{1/n}$ and $C$ depends on $\omega,p,\delta,A$ and an upper bound for $\Vert e^{A\p -u}\Vert_{L^q(X)}$. Now, choose $t_{\infty}>0$ such that 
$$
2C^{2/n}e^{-At_{\infty}/n}\vol(X)^{2/n}<1.
$$
We then have $H(t_{\infty}+1) <(2B)^{-1}$. It then follows from Lemma \ref{lem: Kol tech} that
$$
\f \geq \p -t_{\infty}-3. 
$$
\end{proof}

\subsection{The $\Cc^2$ estimate}
We recall the following well-known result (see \cite{Yau} for a proof).
\begin{lem}\label{lem: AM-GM}
Let $\alpha, \beta$ be positive $(1,1)$-forms. Then
$$
n \left(\frac{\alpha^n}{\beta^n}\right)^{\frac{1}{n}}\leq \tr_{\beta}(\alpha)\leq n \left(\frac{\alpha^n}{\beta^n}\right)\cdot \left( \tr_{\alpha}(\beta)\right)^{n-1}. 
$$
\end{lem}
\noindent Recall that $(\f_t)_{0\leq t<T_{\max}}$ is a smooth solution of the parabolic Monge-Amp\`ere equation (\ref{eq: parabolic}) starting from a smooth strictly $\omega$-psh function $\f_0$.\\

\noindent We use $\Delta_t$ to denote the Laplacian with respect to the K\"ahler metric $\omega_t$ which is defined by
$$
\omega_t:=\theta_t + dd^c \f_t.
$$

\noindent In order to establish the $\Cc^2$ estimates we need a lower bound for $\dot{\f_t}$. Let us stress that during the proof of our estimates, the constants $C_1,C_2...$ stand for various positive constants which are under control. 

\begin{prop}\label{prop: bound below dot}
For all $(t,x)\in (\vep,T]\times X$,
$$
\dot{\f}_t(x)\geq n\log (t-\vep)+ A(\Psi_t-\f_t)(x) -C,
$$
where  $A, C$ are positive constants depending on $\vep$, an upper bound for $E_1,E_2$ and our fixed parameters. Here, $\Psi_t$ is defined by
$$
\Psi_t=\left(1-\frac{t}{2S}\right)\p.
$$ 
\end{prop}
\begin{proof}
Consider the following  function $G$ defined on $(\vep,T]\times X$ by
$$
G(t,x):= \dot{\f}_t(x) +A(\f_t-\Psi_t) - n\log (t-\vep), \ \ \Psi_t:= \left(1-\frac{t}{2S}\right)\p,
$$
where $A$ is a positive constant to be specified later.
Observe that $G$ is smooth on $(\vep,T]\times X_0$, where
$$
X_0:=\{x\in X \setdef \p(x)>-\infty\}.
$$ 
In the sequel, all computations take place in $(\vep,T]\times X_0$. We compute 
$$
\left(\frac{\partial}{\partial t}-\Delta_t \right) G =   A\dot{\f}_t - \frac{n}{t-\vep} +A\frac{\p}{2S} -nA +  A\tr_{\omega_t}(\theta_t +dd^c \Psi_t) +\tr_{\omega_t}(\chi). 
$$
It follows from
$$
\theta_{t} \geq \left(1-\frac{t}{3S}\right)\omega \ \ \mbox{and}\ \  dd^c \Psi_t+ \left(1-\frac{t}{2S}\right)\omega \geq 0  
$$
that 
$$
\theta_t + dd^c \Psi_t \geq \frac{t\omega}{6S}\geq \frac{
\vep \omega}{6S} .
$$
We choose $A>0$ big enough so that 
$$
\frac{A\vep \omega}{6S} + \chi \geq \omega.  
$$
Thus we have
\begin{equation}
\label{eq: C2 est 1}
\left(\frac{\partial}{\partial t}-\Delta_t \right) G \geq   A\dot{\f}_t - \frac{n}{t-\vep} -nA + A\frac{\p}{2S} +  \tr_{\omega_t}(\omega). 
\end{equation}
It follows from Lemma \ref{lem: AM-GM} that 
\begin{equation}
\label{eq: C2 est 2} 
\tr_{\omega_t}(\omega) \geq n \left(\frac{\omega^n}{\omega_t^n}\right)^{1/n} = n\exp{\left(\frac{-\dot{\f_t}}{n}\right)}.
\end{equation}
It is elementary that for each small constant  $b>0$ and each big constant $B>0$  there exists a constant $B'$ depending only on $b$ and $B$ such that
$$
b e^{x} -Bx \geq- B', \ \ \forall x\in \R.
$$
Using this observation, we can find positive constants $C_1, C_2$ under control such that 
\begin{equation}
\label{eq: C2 est 3}
A\dot{\f}_t + n \exp{\left(\frac{-\dot{\f_t}}{n}\right)} \geq \exp{\left(\frac{-\dot{\f_t}}{n}-C_1\right)}-C_2.
\end{equation}
From (\ref{eq: C2 est 1}), (\ref{eq: C2 est 2}) and  (\ref{eq: C2 est 3}) we  deduce   that  
\begin{equation}
\label{eq: C2 est 4}
\left(\frac{\partial}{\partial t}-\Delta_t \right) G \geq   \exp{\left(\frac{-\dot{\f_t}}{n}-C_1\right)} + A\frac{\p}{2S}  - \frac{n}{t-\vep} -C_2 -An. 
\end{equation}
Notice that the function $G$ is smooth on $(\vep,T]\times X_0$ and it goes to $+\infty$ on the boundary. Thus we see that $G$ attains its minimum on $(\vep,T]\times X_0$ at some point $(t_0,x_0)\in  (\vep,T]\times X_0$. It follows from the minimum principle and (\ref{eq: C2 est 4}) that, at $(t_0,x_0)$, we have
$$
\dot{\f}_{t} \geq -n \log \left( C_2 + \frac{n}{t_0-\vep} - A\frac{\p}{2S} +An\right) -n C_1.
$$
We then get 
\begin{eqnarray*}
G(t_0,x_0) &\geq & -C_3 - n\log \left[C_5(t_0-\vep)+ n-\frac{A(t_0-\vep)\p(x_0)}{2S} \right] + A(\f_{t_0}-\Psi_{t_0})\\
&\geq & -C_4 - n\log \left[C_5(t_0-\vep)+ n-\frac{A(t_0-\vep)\p(x_0)}{2S}\right] - \delta\p(x_0),
\end{eqnarray*}
where $\delta= \frac{At_0(S-T)}{2TS}>0$ and the last inequality follows from Theorem \ref{thm: C0-estimate}. We then obtain a uniform lower bound for $G(t_0,x_0)$, and hence the result follows.  
\end{proof}

\begin{thm}
\label{thm: C2 est}
Fix $\vep>\vep_0$. For all $x\in X$ and $t\in [\vep,T]$ one has
$$
0\leq (t-\vep)\log \tr_{\omega}(\omega_{t})\leq  -A \p + C,
$$
where $A,C$ are  positive constants depending on $\vep$, an upper bound for $E_1,E_2$ and our fixed parameters.    
\end{thm}


\begin{proof}
Consider 
$$
H:= (t-\vep)\log \tr_{\omega}(\omega_t)-A(\f_{t}-\Psi_t), t\in [\vep,T] , x\in X,
$$
where $A$ is a positive constant to be specified later and the function $\Psi_t$ is defined by 
$$
\Psi_t:= \left(1-\frac{t}{2S}\right)\p.
$$  
Notice that $H$ is smooth on $[\vep,T]\times X_0$ and $H$ goes to $-\infty$ on the boundary of $X_0$, where 
$$
X_0:=\{x\in X \setdef \p(x)>-\infty\}.
$$
Set $u=\tr_{\omega}(\omega_t)$. We compute
$$
\left(\frac{\partial }{\partial t}-\Delta_t \right)H= \log u +\frac{t-\vep}{u}\frac{\partial u}{\partial t}-
A\dot{\f}_{t}-A\frac{\p}{2S}- (t-\vep) \Delta_t \log u + A \Delta_t (\f_{t}- \Psi_t).
$$
As in the proof of Proposition \ref{prop: bound below dot}  we have
$$
\tr_{\omega_t}(\theta_t + dd^c \Psi_t) \geq \frac{\vep}{6S} \tr_{\omega_t}(\omega). 
$$
Thus we obtain
$$
\Delta_t( \f_t-\Psi_t)\leq n- \frac{\vep}{6S}\tr_{\omega_t}(\omega).
$$
Thanks to \cite[Lemma 2.2]{CGP13} (which is an improvement of \cite{Siu87}) we also have
$$
\Delta_t \log u \geq \frac{\tr_\omega (dd^c \dot{\f_t})}{\tr_\omega(\omega_{t})}-B\,\tr_{\omega_{t}}(\omega),
$$
where $B$ depends only on a lower bound for the holomorphic bisectional curvature of $\omega$. Furthermore,
$$
\frac{t-\vep}{u}\frac{\partial u}{\partial t} =\frac{t-\vep}{u} \left(\tr_\omega (\chi)+ \tr_\omega (dd^c \dot{\f_t}) \right) \leq \frac{t-\vep}{u}\tr_\omega (dd^c \dot{\f_t})+C_1(t-\vep) \,\tr_{\omega_{t}}(\omega)
$$
where in the last inequality we use the fact that $\tr_\omega (\chi)$ is bounded from above together with the trivial inequality $n\leq \tr_{\omega_{t}}(\omega)\tr_{\omega}(\omega_{t})$.
Thus
$$
\frac{t-\vep}{u}\frac{\partial u}{\partial t}-(t-\vep)\Delta_t \log u \leq (B+C_1)(t-\vep) \,\tr_{\omega_{t}}(\omega).
$$
It follows from Lemma \ref{lem: AM-GM} applied to $\alpha=\omega_{t}$ and $\beta=\omega$ that
$$
\tr_{\omega}(\omega_{t})\leq n\, e^{\dot{\f_t}}\left( \tr_{\omega_{t}}(\omega)\right)^{n-1}.
$$
Using the inequality $(n-1)\log x < x+C_n$, we obtain
$$\log u \leq \log n  + \dot{\f_t} +(n-1)\log  \tr_{\omega_{t}}(\omega) \leq \dot{\f_t}+\tr_{\omega_{t}}(\omega)+C_2.
$$
Thus we have
$$
\left(\frac{\partial }{\partial t}-\Delta_t \right)H \leq -(A-1)\dot{\f}_{t}+\left[(B+C_1)(t-\vep)+1-\frac{A\vep}{6S}\right]\tr_{\omega_{t}}(\omega)-\frac{A\p}{2S}+C_3.
$$
We will choose $A>0$ big enough such that 
$$
(B+C_1)T+1-\frac{A\vep}{6S}\leq -1.
$$ 
Assume now that $H$ achieves its maximum at some point $(t_0,x_0)\in  [\vep,T]\times X_0$. If $t_0=\vep$ then we are done. Assume that $t_0>\vep$.  From now on our computations take place at this maximum point. 
By the maximum principle we have
$$
0\leq \left(\frac{\partial }{\partial t}-\Delta_t \right)H \leq C_3-(A-1)\dot{\f_t} -\frac{A\p}{2S}- \tr_{\omega_t}(\omega),
$$
or, equivalently
\begin{equation}
\label{eq: C2 est 5}
\tr_{\omega_t}(\omega)\leq C_3-(A-1)\dot{\f_t}- \frac{A\p}{2S}.
\end{equation}
In particular, since $\tr_{\omega_t}(\omega)\geq 0$, we have the following upper bound for $\dot{\f_t}$:
\begin{equation}\label{eq: C2 est 6}
\dot{\f_t}\leq \frac{C_3}{A-1} - \frac{A\p}{2S(A-1)}.
\end{equation}
It follows from Lemma \ref{lem: AM-GM} that 
$$
\tr_{\omega_t}(\omega) \geq n \exp\left(\frac{-\dot{\f_t}}{n}\right).
$$
Plugging this into (\ref{eq: C2 est 5}) we obtain
\begin{equation}
\label{eq: C2 est 7}
\tr_{\omega_t}(\omega)\leq C_4- \frac{A\p}{4S}.
\end{equation}
Thus applying again Lemma \ref{lem: AM-GM} yields
$$
 \log u  \leq \log n + \dot{\f_t} + (n-1)\log \left[C_4 - \frac{A\p}{4S}\right]. 
$$
From this and from (\ref{eq: C2 est 6}) we obtain
$$
H \leq C_5 -A\left[\f_t -\left(1-\frac{t}{2S}-\frac{t-\vep}{2(A-1)S} \right)\p\right] 
+(n-1)(t-\vep)\log \left[C_4-\frac{A\p}{4S} \right].
$$
Now, we increase $A$ a little bit, if necessary,  so that 
$$
\delta:= \frac{\vep}{2T}-\frac{\vep}{2S}-\frac{T}{2(A-1)S}>0 .
$$
Thus since  $\p\leq 0$ we obtain
$$
H \leq C_5 -A\left[\f_t -\left(1-\frac{t}{2T} \right)\p\right] + A\delta \p
+(n-1)(t-\vep) \log \left[C_4 - \frac{A\p}{4S}\right].
$$
The second term is uniformly bounded from above thanks to Theorem \ref{thm: C0-estimate}. Now, using the obvious inequality $-\vep x + \log x \leq C$  we obtain a uniform upper bound for $H$ at the point 
$(t_0,x_0)$ and the result follows. 
\end{proof}

\subsection{More estimates}
In this subsection we derive some other a priori  estimates for the solutions of the parabolic equation (\ref{eq: parabolic}). These estimates will be used later to prove stronger convergence of the flow at zero when the initial data has some regularity properties (Theorems \ref{thm: convergence at zero smooth case} and \ref{thm: cvg zero exp}). The setting and the notations will be the same as those in the previous section. 
\begin{prop}
\label{prop: est dot more}
Assume that there exist smooth $\omega$-plurisubharmonic functions $\phi_1,\phi_2$ and  positive constants $0<\delta<1/2$, $C_1>0$ such that  
$$
\dot{\f_0} \geq C_1 \phi_1 \ \ {\rm and} \ \ \f_0 \geq \delta \phi_2.
$$
Then there exists a uniform constant $C_2$ depending on $C_1$ and our fixed parameters such that 
$$
\dot{\f_t} \geq C_2(\phi_2 + 1) + C_1 \phi_1.  
$$
\end{prop}
\begin{proof}
Consider the following function
$$
H(t,x) := \dot{\f_t} -C_1 \phi_1 + A (\f_t - \delta \phi_2),
$$
where $A>0$ is a big constant to be specified later. 
We use the same arguments as in Proposition \ref{prop: bound below dot}. Since $H$ is smooth on $[0,T]\times X$ it attains a minimum at some point $(t_0,x_0)$. If $t_0$ is zero then we are done since $H(0,.)\geq 0$ by our construction. 

Assume now that $t_0>0$. We do our computation below at this minimum point $(t_0,x_0)$.  By the minimum principle we have 
$$
0\geq \left(\frac{\partial}{\partial t} - \Delta_t \right) H  \geq -An  + A \dot{\f_t} + \left(\frac{A}{6}-C_1-A_1\right)\tr_{\omega_t}(\omega). 
$$
Here we use the fact that $\theta_t\geq 2\omega/3$ and $\chi \geq -A_1\omega$, for some constant $A_1$ under control. Now, choose $A$ big enough such that $\frac{A}{6}-C_1-A_1>1$ and using Lemma \ref{lem: AM-GM} combined with the inequality $e^x\geq B x -C_B$, we obtain a  uniform lower bound for $\dot{\f}_{t_0}(x_0)$. The lower bound for $H(t_0,x_0)$ then follows observing that, by \cite[Lemma 2.9]{GZ13}, $\f_t\geq \f_0-c(t)$, and so $\f_t-\delta \phi_2\geq -c(t)$ where $c(t)\rightarrow 0$ as $t\rightarrow 0$.

\end{proof}

\begin{thm}
\label{thm: C2 est more}
Assume that there exist smooth $\omega$-plurisubharmonic function  $\phi_1,\phi_2$ and a positive constants $0<\delta<1/2$ such that 
$$
\Delta_{\omega}\f_0 \leq e^{-C_1\phi_1} \ \ \ \ {\rm and}\ \ \f_0\geq \delta \phi_2. 
$$ 
Then there exists a uniform constant $C_2>0$ such that
$$
\tr_{\omega} \omega_t \leq C (1+  e^{-C_1\phi_1-\delta \phi_2}), \ \forall t\in [0,T].
$$
\end{thm}
\begin{proof}
Consider the function 
$$
H(t,x)= \log \tr_{\omega} (\omega_t) + C_1\phi_1 - A(\f_t-\delta\phi_2),
$$
where $A>0$ is a big constant (under control) to be specified later. This is a smooth function on $[0,T]\times X$ and hence it attains its maximum at some point $(t_0,x_0)$. If $t_0=0$ we are done since by our construction $H(0,.)\leq 0$. Assume now that $t_0>0$. From now on our computations take place at this maximum point $(t_0,x_0)$.  By the maximum principle and by the same arguments as in Theorem \ref{thm: C2 est} we eventually get 
\begin{equation}
\label{eq: C2 est more 1}
0\leq \left(\frac{\partial }{\partial t}-\Delta_t \right)H \leq A_1-A\dot{\f_t} - \left(\frac{A}{6}-B-C_1-A_2\right)\tr_{\omega_{t}}(\omega),
\end{equation}
where $-B<0$ is a lower bound for the holomorphic bisectional curvature of $\omega$, and $A_2$ is a constant under control.  
Here we use the fact  that $\delta <1/2$ and $\theta_t \geq \frac{2}{3} \omega$ for all $t\in [0,T]$ and 
$$
-\tr_{\omega_t}(dd^c \phi_1) \leq \tr_{\omega_t}(\omega), \ \ \tr_{\omega_t}(\chi) \geq -A_2\tr_{\omega_t}(\omega). 
$$
It follows from Lemma \ref{lem: AM-GM} that 
$$
\tr_{\omega_t} \omega \geq n \exp\left(\frac{-\dot{\f_t}}{n}\right).
$$
Plugging this into (\ref{eq: C2 est more 1}), choosing $A$ big enough and using the inequality $-\varepsilon x +\log x \leq C_\varepsilon$ with $0<\varepsilon<(An)^{-1}$, we obtain an upper bound for $\tr_{\omega_t}(\omega)$ at $(t_0,x_0)$. We infer that there exists a uniform constant $C>0$ such that $H(t_0,x_0)\leq C$.  Thus $H(t,x)\leq C$ and so
$$\log \tr_{\omega}(\omega_t) \leq C-C_1\phi_1 +A\delta \phi_2$$ or equivalently
$$\tr_{\omega}(\omega_t) \leq C' e^{-C_1\phi_1 +A\delta \phi_2}.$$
\end{proof}

\section{Regularity of weak solutions}
\label{sec: Regularity}
In this section we study short time regularity of the K\"ahler-Ricci flow (\ref{eq: KRflow}) starting from a current having positive Lelong number. In particular, we give a proof of Theorem A and show that the result is optimal. 

In the whole section we assume that $\f_0$ is a $\omega$-psh function on $X$ such that
$$
\frac{1}{2c(\f_0)}<T_{\max}. 
$$
This condition guarantees that the approximating smooth solutions will not converge uniformly to $-\infty$ (see \cite[Lemma 2.9]{GZ13}).

\subsection{Construction of the maximal weak solution}
\label{construction}
For convenience, let us recall the construction of the maximal weak solution due to Guedj and Zeriahi \cite{GZ13}. Let $(\f_{0,j})$ be a decreasing sequence of smooth strictly $\omega$-plurisubharmonic functions which converges to $\f_0$. Let $(\f_{t,j})$ be the sequence of smooth solutions of the equation (\ref{eq: parabolic}) starting from $\f_{0,j}$. When $j$ goes to $+\infty$, for each $t>0$ fixed, the sequence $(\f_{t,j})$ decreases to a well-defined $\theta_t$-psh function $\f_t$ (not identically $-\infty$ thanks to \cite[Lemma 2.9]{GZ13}).  

\begin{ex}\label{CondTmax}
 Let us stress that the condition $1/2c(\f_0)<T_{\max}$ is necessary to insure that the maximal flow is well-defined as the following example shows. \\
\indent Consider $X=\mathbb{P}^1$ and $\omega=4\omega_{FS}$ whose potential is given in homogeneous coordinate by
$$
\rho = 2\log (|z|^2+|w|^2).
$$ 
It is well-known that $\Ric(\omega)= \omega$. Consider the K\"ahler-Ricci flow 
$$
\dot{\omega}_t = - \Ric \ \omega
$$
and observe that in this case we have $T_{\max}=1$. Consider the following sequence of functions
$$
\f_{0,j}:= 2\log (|z|^2+ |w|^2/j) -  \rho. 
$$
The sequence $(\f_{0,j})_j$ decreases to a $\omega$-psh function $\f_0$ with $c(\f_0)=1/4$. Thus $1/2c(\f_0)=2>T_{\max}$. A straightforward computation shows that 
$$
\omega +dd^c \f_{0,j} = e^{-\f_{0,j} - \log j} \omega. 
$$
Let us define
$$
\f_{t,j}:= (1-t)\f_{0,j} - t\log j - t + (t-1)\log (1-t),\ \ t\in (0,1) 
$$
and note that $\f_{t,j}$ solves the following parabolic Monge-Ampere equation 
$$
\omega - t\Ric(\omega) + dd^c u_t = e^{\dot{u}_t} \omega,
$$
with initial data $\f_{0,j}$. We can also check that $\f_{t,j}$ decreases to $-\infty$ if $t>0$. Hence the maximal solution constructed in \cite{GZ13} is not well-defined. Note, however, that the flow $\omega_{t,j}$ converges to $\omega_t=(1-t)\omega_0$.
\end{ex}

\begin{lem}
\label{lem: independent}
The limit $\f_t$ obtained as above does not depend on the choice of the approximating sequence $(\f_{0,j})$. 
\end{lem}
\begin{proof}
Let $\p_{0,j}$ be another decreasing sequence of smooth strictly $\omega$-psh functions which converges to $\f_0$. Let $(\p_{t,j})$ be the corresponding smooth solutions of (\ref{eq: parabolic}) starting from $\p_{0,j}$. Fix $k\in \N$ and $\vep>0$. By Hartogs' lemma there exists $j_k\in \N$ such that 
$$
\p_{0,j} \leq \f_{0,k} + \vep, \ \ \forall j\geq j_k. 
$$
It follows from the comparison principle (Proposition 
\ref{prop: clas_comp_prin}) that 
$$
\p_{t,j} \leq \f_{t,k} + \vep, \ \ \forall j\geq j_k.
$$
Now, letting $k\to +\infty$ and then $\vep \to 0$ we get one inequality. The result follows since we can inverse the role of $\f_t$ and $\p_t$ in the above arguments. 
\end{proof}
\begin{rem}
\label{rem: comparison maximal}
The proof given above also shows that if $\f_t$ and $\p_t$ are two maximal solutions of (\ref{eq: parabolic}) with $\f_0\leq \p_0$ then $\f_t\leq \p_t$.  
\end{rem}

\subsection{Short time singularity} It was shown in \cite[Theorem 6.1]{GZ13} that when $t>1/2c(\f_0)$, $\omega_t:=\theta_t+dd^c\f_t$ is a K\"ahler form satisfying the twisted K\"ahler-Ricci flow equation (\ref{eq: KRflow}). In particular they proved that for $t>1/2c(\f_0)$ $\f_t$ is a smooth and verifies equation (\ref{eq: parabolic}). \\

In what follows we investigate the behavior of the maximal solution in short time. We show that the maximal solution of the equation (\ref{eq: parabolic}) starting from a current with positive Lelong numbers is unbounded for short time. This relies on a quite simple observation that we will explain in the next lines before entering into the details of the proof of Lemma \ref{lem: singularity}. 

Assume that $\f_0\in \psh(X,\omega)$ has positive Lelong number at some $x_0\in X$, then   
$$
\f_0 \leq \gamma \log |z| + C
$$
for some positive constants $\gamma, C$. Without loss of generality we can assume that $\f_0= \phi(z)\log |z|^2$, where $z$ is a local coordinate in a neighborhood $V$ around $z=0$ and $\phi$ is a cut-off function which vanishes outside $V$ and $\phi\equiv 1$ near $z=0$. We can also assume that $\f_0\in \psh(X,\omega/2)$ and $T_{\max}>1$. Let $\f_t$ be the maximal solution of the equation (\ref{eq: parabolic}) with initial data $\f_0$. Fix $\vep>0$ small enough. Consider the following family 
$$
u_{t,\vep}:= \left[1-(n+1)t\right] \phi(z)\log (|z|^2+\vep^2) + Ct.
$$
We can choose $C>0$ large enough (which does not depend on $\vep$) such that $u_{t,\vep}$ is a supersolution of the equation (\ref{eq: parabolic}). Since $u_{0,\vep}\geq \f_0$ it follows from the comparison principle that $u_{t,\vep} \geq \f_t, \ \forall t\in (0,1/(n+1))$. By letting $\vep\searrow 0$ we can conclude that $\f_t$ is unbounded and has positive Lelong number for very short time. 

\medskip

In the Lemma below we prove that the solution is singular for all $0<t<1/2nc(\f_0)$. When $n=1$ this coupled with \cite[Theorem 6.1]{GZ13} gives a complete picture of the regularity properties of the K\"ahler-Ricci flow. 


\begin{lem}
\label{lem: singularity}
Let $\phi\in \psh(X,\omega)$ be such that $e^{\gamma\phi}\in \Cc^{\infty}(X)$ for some positive constant $\gamma$. Assume that $\f_0\leq \phi\leq 0$. Then there exists a positive constant $C$ depending on an upper bound for  $dd^c e^{\gamma \phi}$ such that 
 $$
 \f_t \leq (1-n\gamma t) \phi + Ct , \ \forall t\in [0,1/n\gamma]. 
 $$
\end{lem}
\begin{proof}
Notice that in the statement of the lemma we impose implicitly that $1/n\gamma <T_{\max}$. For simplicity we also assume that 
$$
\theta_t \leq \omega , \  \forall t\in [0, 1/n\gamma]. 
$$
In practice $\gamma$ will be very large and these conditions are irrelevant.
By replacing $\phi$ by the function $\gamma^{-1}\log ( e^{\gamma \phi}+\vep)$ and letting $\vep \to 0$ we can assume that $\phi$ is actually smooth on $X$. Exploiting the same arguments in the proof of Lemma \ref{lem: independent} we can also assume that $\f_0$ is smooth. 

Observe that there is a uniform constant $C>0$ depending only on an upper bound for  $dd^c e^{\gamma \phi}$ such that
$$
 dd^c \phi  \leq C e^{-\gamma \phi}\omega. 
$$
Set 
$$
\phi_t := (1-n\gamma t)\phi + t\log (C^n2^n). 
$$
Since $\phi\leq 0$ and $dd^c \phi \leq Ce^{-\gamma \phi} \omega$ we see that
$$
0\leq (\omega + dd^c \phi_t)  \leq (1+ Ce^{-\gamma \phi})\omega \leq 2Ce^{-\gamma \phi} \omega. 
$$ 
It then follows that 
$$
(\omega +dd^c \phi_t)^n \leq 2^n C^n e^{-n \gamma \phi} \omega^n=e^{\dot{\phi}_t}\omega^n. 
$$
Then $\phi_t$ is a supersolution of the following parabolic equation 
$$
(\omega +dd^c u_t)^n = e^{\dot{u}_t} \omega^n,
$$
while $\f_t$ is a subsolution since we have assumed that $\theta_t \leq \omega$.  The result follows by the classical maximum principle.
\end{proof}

Note that the above result shows that for short time we have propagation of sigularities (as it is highlighted in Theorem \ref{thm: singularity} below) and that our $C^0$ estimate in Theorem \ref{thm: C0-estimate} is optimal.
\begin{thm}
\label{thm: singularity}
Assume that $\f_0$ has positive Lelong numbers. Then we have the following relation between the Lelong number superlevel sets of $\f_t$ and $\f_0$:
$$
D_c(\f_0) \subset D_{c(t)}(\f_t), \ \ c(t)=c - 2nt. 
$$
In particular,  the maximal solution $\f_t$ has positive Lelong numbers for any $t<1/2nc(\f_0)$. 
\end{thm}
\begin{proof} Let $c>0$ and   $x_0\in D_c(\f_0)$. It follows from the very definition of the Lelong number that we can find a function $\phi\in \psh(X,A\omega)$ for some positive constant $A$ such that  
$$
e^{2\phi/c} \in \Cc^{\infty}(X)\ \ {\rm and}\ \  \phi \geq \f_0. 
$$
In fact, the function $\phi$ can be written locally as  
$$
\phi(z)=c\log |z|,
$$
where $z$ is a local coordinate centered at $x_0$. Applying Lemma \ref{lem: singularity} we get 
$$
\f_t \leq (1-2nt/c)\phi + Ct.
$$
It follows that $\nu(\f_t,x_0)\geq c - 2nt$, hence $x_0\in D_{c(t)}(\f_t)$. 

In particular, if $t<1/2nc(\f_0)$ then for a constant $c$ such that $2nt< c<1/c(\f_0)$ the set 
$D_{c}(\f_0)$ is non empty as follows from Skoda's integrability theorem since $e^{-\frac{2}{c} \f_0}\notin L^1$. It follows from the first part of the theorem that $D_{c(t)}(\f_t)$ is nonempty for $c(t)=c-2nt>0$. 
\end{proof}

Our previous analysis describes various singular situations. It also sheds some light on the behavior of the maximal solution in short time. In fact, it strongly suggests that the maximal solution is smooth outside analytic subsets of $X$. We next prove Theorem A  confirming this suggestion by employing our previous a priori estimates.

\subsection{Proof of Theorem A}
Let $(\f_{0,j})$ be a decreasing sequence of smooth strictly $\omega$-psh functions converging to $\f_0$. This is always possible thanks to Demailly's regularization (see \cite{Dem94}, or \cite{BK07}). Let $(\f_{t,j})$ be smooth solutions of the parabolic equation (\ref{eq: parabolic}) starting from $\f_{0,j}$. It was proven by Guedj-Zeriahi  in \cite{GZ13} that, as $j\to+\infty$, the family $\f_{t,j}$ decreases to $\f_t$, which are $\theta_t$-psh functions on $X$. They also proved that for each $t>(2c(\f_0))^{-1}$, the function $\f_t$ is smooth on $X$ and satisfies (\ref{eq: parabolic}) pointwise on $X$. Moreover, 
$\f_t \to \f_0$
in capacity as $t\to 0$. 

We investigate the regularity of $\f_t$ for small $t$. Let us fix $\vep_0>0$. 
Fix a constant $\vep>\vep_0$.  Thanks to Lemma \ref{lem: Dem} we  fix $\p\in \psh(X,\omega)\cap \Cc^{\infty}_{\rm loc}(X\setminus D_{\kappa})$ (the constant $\kappa=\kappa(\vep_0)$ is defined  by Lemma \ref{lem: Dem})  such that 
\begin{equation}\label{eq: proof A 1}
\int_X \exp{\left(\frac{2\p-2\f_0}{\vep_0}\right)} dV <+\infty.
\end{equation}
We can always assume that $\p\leq 0$.  Recall that by the choice of $T$ and since $\p$ is less singular that $\f_0$, we always have
\begin{equation}
\label{eq: proof A 2}
\int_X \exp{\left(\frac{-\p}{T}\right)}\omega^n <+\infty. 
\end{equation}
We will prove that $\f_t$ is smooth on $X\setminus D_{\kappa}$ for each $t>\vep$. 
Fix a compact subset 
$K\Subset X\setminus D_{\kappa}$. We apply Lemma \ref{lem: upper bound varphi_t} and Theorem \ref{thm: C0-estimate} for the flow $\f_{t,j}$ (and use the function $\p$ above) and note that, thanks to (\ref{eq: proof A 1}) and (\ref{eq: proof A 2}) the constants in these theorems do not depend on $j$. Thus
$$
\sup_{t\in [\vep,T]}\sup_K |\f_{t,j}| \leq C(\vep,K).
$$
It follows from Theorem \ref{thm: C2 est} that we also have a uniform bound on  $\Delta \f_{t,j}$:
$$
\sup_{t\in [\vep,T]}\sup_K |\Delta_{\omega} ( \f_{t,j})| \leq C(\vep,K).
$$
Now, from Proposition \ref{prop: bound below dot} and from the fact that the functions $\f_{t,j}$ satisfy the equation (\ref{eq: parabolic}) we obtain a locally uniform bound on $\dot{\f}_{t,j}$:
$$
\sup_{t\in [\vep,T]}\sup_K  \left|\frac{d \f_{t,j}}{dt} \right| \leq C(\vep,K).
$$
Then using the complex parabolic Evans-Krylov theory together with 
Schauder's estimate (see \cite[Theorem 4.1.4]{BEG13}), we obtain higher order estimates on $K\times[\vep,T]$ just as in \cite{GZ13}. This justifies the smoothness of $\f_t$ on $X\setminus D_{\kappa}$.

\section{Uniqueness and stability}
\label{sec: Stability}
\subsection{Currents with zero Lelong numbers}
\label{subsect: zero LeLong}
In this section we study the uniqueness and stability of the K\"ahler-Ricci flow starting from currents having zero Lelong number.  The following estimate was proven in \cite{GZ13} for the maximal flow. The ideas in \cite{GZ13} are the prove such estimates (in a uniform way) for the smooth flows approximating the maximal flow. We prove in the lemma below that, in the case of zero Lelong numbers, the same estimate holds for any weak solution. The idea of our proof is quite simple: we follow the same strategy of Guedj-Zeriahi, but instead of dealing with the flow $\f_t$ itself we deal with the flow $\f_{t+\vep}$ (for $\vep>0$) and then we let $\vep\to 0$. The key point is that the new flow is smooth {\it up to zero} and the classical comparison principle can be applied. 

\begin{lem}
\label{lem: lower bound on varphi t}
Assume that $(\varphi_t)$ is a weak solution of the parabolic equation (\ref{eq: parabolic}) with initial data $\f_0$. Assume also that 
$c(\f_0)=+\infty$. Let $\beta>0$ be a constant such that $2\beta >1/T_{\max}$. Then 
$$
\varphi_t \geq (1- 2\beta t)\varphi_0 - C(t),\ \  0< t <(2\beta)^{-1}. 
$$
where $C(t)$ depends only on $t$ and $C(t)\searrow 0$ as $t\searrow 0$. 
\end{lem}

\begin{proof}
Fix a positive constant  $\alpha$  such that  $2\beta-\alpha>1/T_{\max}$. Take a small constant $\vep>0$ such that 
$$
(2\beta -\alpha)\omega + (1+2\beta\vep)\chi >0.
$$
Set
$$
u_t:=\f_{t+\varepsilon}, \ t\in [0,T_{\max}-\vep).
$$ 
Since $c(\f_0)=+\infty$, it follows from \cite[Theorem 1.1]{GZ13} that $u$ is smooth on $X\times [0,T]$.  Applying \cite[Lemma 2.9]{GZ13} we see that  for any $x\in X$ and any 
$t\in \left[0, (2\beta)^{-1}-\vep\right)$,
\begin{equation}
\label{eq: convergence in capacity}
(1-2\beta t)u_0(x)+ \alpha tv_{\vep}(x) +n (t\log t-t)\leq u_t(x),
\end{equation}
where $v_{\vep}$ is the unique continuous $\omega$-psh function satisfying 
$$
\alpha^n(\omega +dd^c v_{\vep})^n = e^{\alpha v_{\vep}-2\beta \f_{\vep}}\omega^n.
$$
Since $\f_0$ has zero Lelong number everywhere and since $\f_{\vep}$ converges to $\f_0$ in $L^1$ as $\vep \searrow 0$ it follows from  \cite[Theorem 3.1]{Zer01} that 
$$
\int_X e^{-4\beta\f_{\vep}} \omega^n \leq C,
$$
for a uniform constant $C>0$. Therefore, it follows from 
Ko{\l}odziej's $\Cc^0$ estimate (see \cite{Kol98}) that $v_{\vep}$ is uniformly bounded. Thus from (\ref{eq: convergence in capacity}) we get
$$
(1-2\beta t)\f_{\vep}(x)-C(t)\leq \f_{t+\vep}(x).
$$
Letting $\vep\rightarrow 0$ we obtain the result. 
\end{proof}
From Lemma \ref{lem: lower bound on varphi t} we immediately get the following convergence result.
\begin{cor}
Assume that $c(\f_0)=+\infty$ and $\f_t$ is a weak solution of the parabolic equation (\ref{eq: parabolic}). Then $\f_t$ converges to $\f_0$ in capacity as $t$ goes to zero. 
\end{cor}

We are now in the position to prove the following comparison principle.  

\begin{thm}
\label{thm: comparison zero}
Let $\f_0, \p_0\in \psh(X,\omega)$ be such that $c(\f_0)=+\infty$. Let $(\f_t), (\p_t)$ be weak solutions of the equation (\ref{eq: parabolic}) starting from $\f_0$ and $\p_0$ respectively. Then 
$$
\f_0\leq \p_0 \Longrightarrow \f_t\leq \p_t, \forall t>0.
$$  
In particular, if $c(\f_0)=+\infty$ then there is a unique weak solution of the equation (\ref{eq: parabolic}) with initial data $\f_0$. 
\end{thm}
\begin{proof}
We devide the proof into two steps. 

\noindent{\bf Step 1.} We  assume that $\chi\geq 0$. Fix $\vep>0$ and consider the following function
$$
v_t:= \p_{t+\vep} + C(\vep), 
$$
where $C(\vep)$ is the constant defined in Lemma \ref{lem: lower bound on varphi t}. Accordingly we have  that $v_0\geq \p_0$. We compute 
$$
\dot{v}_t = \dot{\p}_{t+\vep} = \log \left[\frac{(\omega+ (t+\vep)\chi + dd^c v_t)^n}{\omega^n}\right].
$$
Since $\chi\geq 0$ and $\f_t$ is a weak solution of (\ref{eq: parabolic}) it follows that
$$
\dot{\f}_t = \log \left[\frac{(\omega+ t\chi + dd^c \f_t)^n}{\omega^n}\right] \leq \log \left[\frac{(\omega+ (t+\vep)\chi + dd^c \f_t)^n}{\omega^n}\right].
$$
Our computations above show that  $v_t$ is a smooth solution of the following parabolic equation
$$
\dot{v}_t = \log \left[\frac{(\omega+ (t+\vep)\chi + dd^c v_t)^n}{\omega^n}\right],
$$ 
while $\f_t$ is a weak subsolution. Since $\f_0\leq \p_0\leq v_0$
it follows from the generalized comparison principle (Proposition \ref{prop: gen comp prin}) that $\f_t\leq v_t$. Since $C(\vep)\searrow 0$ as $\vep\searrow 0$ the result follows.

\medskip

\noindent{\bf Step 2.} We remove the non-negative assumption on $\chi$. The idea is to make a change of variable to get rid of the sign of $\chi$. Without loss of generality we can assume that $T_{\max}>1$ and hence $\chi > -\omega$. We make the following change of variable
$$
u_t:= e^{t}\f_{1-e^{-t}}. 
$$
Then $u_t$ satisfies the following equation
$$
\dot{u}_t = u_t -nt + \log  \left[\frac{(\omega+ (e^{t}-1)\chi' + dd^c u_t)^n}{\omega^n}\right],
$$
where $\chi'=\omega+\chi >0$. Thus this change of variable gives  us another equation with the twisted form $\chi'>0$. We then can apply our techniques in the first step. For the reader's convenience we  explain the proof in  full details. 

Fix $\vep>0$ small enough and consider the following families: 
$$
u_t:=e^{t}\f_{1-e^{-t}}\ , \ v_t:=e^{t+\vep}\p_{1-e^{-t-\vep}} +n\vep (e^t-1) + e^{t+\vep}C(1-e^{-\vep}),
$$
where $C(s)$ is the constant defined in Lemma \ref{lem: lower bound on varphi t}. Observe that $v$ is smooth on $[0,t_0]\times X$, where $t_0$ is a fixed positive constant such that $1-e^{-\vep-t_0}<T_{\max}$. We can assume that $\f_t$ and $\p_t$ are negative for small $t$. Then it follows from Lemma \ref{lem: lower bound on varphi t} that  
$$
v_0\geq \f_0 =u_0. 
$$
Since $\chi':= \chi+\omega\geq 0$ it follows that 
$$
\dot{u_t}\leq u_t-nt+\log \left[\frac{(\omega+ (e^{t+\vep}-1)\chi' + dd^c u_t)^n}{\omega^n}\right],
$$
and 
$$
\dot{v_t} = v_t - nt + \log \left[\frac{(\omega+ (e^{t+\vep}-1)\chi' + dd^c v_t)^n}{\omega^n}\right].
$$
It thus follows from the comparison principle (Proposition \ref{prop: gen comp prin}) that $u_t\leq v_t, \forall t\in [0,t_0]$. Letting $\vep\searrow 0$ we obtain $\f_t\leq \p_t$ for any $t\in (0,t_0)$. Then we can conclude $\f_t\leq \p_t$ for any $t\in (0,T_{\max})$. 
\end{proof}

The following stability result  follows from the uniqueness of the flow.  

\begin{thm}
\label{thm: stability zero}
Let $\f_0$ be a $\omega$-psh function and $\f_{0,j}$ a sequence of $\omega$-psh functions such that $\f_{0,j}\rightarrow \f_0$ in $L^1$. 
Assume also that $c(\f_0)=+\infty$. Denote by $\f_{t,j}$ and $\f_t$  the solutions of (\ref{eq: parabolic}) with initial data $\f_{0,j}$ and $\f_0$ respectively. Then for each $t>0$,
$$
\f_{t,j} \rightarrow\f_t \quad \mbox{in} \;\;\Cc^{\infty}(X) \ \mbox{as}\quad j\to +\infty.
$$
\end{thm}

\begin{proof}
Assume first that the sequence $(\f_{0,j})$ is deceasing. Then we can find a decreasing sequence $(\p_{0,j})$ of smooth strictly $\omega$-psh functions which converges to $\f_0$ and satisfies
$$
\f_0 \leq \f_{0,j} \leq \p_{0,j}, \ \ \forall j\in \N. 
$$
Let $(\f_{t,j})$ and $(\p_{t,j})$ be the corresponding weak solutions of the equation (\ref{eq: parabolic}) with initial data $\f_{0,j}$ and $\p_{0,j}$ respectively. 
It follows from the comparison principle (Theorem \ref{thm: comparison zero}) that 
$$
\f_t \leq \f_{t,j} \leq \p_{t,j}, \ \ \forall j\in \N. 
$$
 By Lemma \ref{lem: independent} we know that $\p_{t,j} \searrow \f_t$, hence $\f_{t,j} \searrow \f_t$. Moreover it follows from \cite{GZ13} that  for each $\vep>0$, $T<T_{\max}$, $k\in \N$ we have
$$
\Vert\f_j\Vert_{\Cc^{k}([\vep, T]\times X)} \leq C(\vep,T,k), 
$$
for $j>j(\vep)$ large enough. Then using Arzel\`a-Ascoli theorem, we know that, for each $t>0$, 
$\f_{t,j}$ converges to $\f_t$ in $\Cc^{\infty}(X)$ as $j$ goes to $+\infty$.  Hence the result follows. 
\medskip

Now, let us prove the general statement. We can assume that $\varphi_0$ is negative. Recall that, for each $j$ fixed, $\f_{t,j}$ converges in $L^1(X)$ to $\f_{0,j}$ as $t\searrow 0$. The arguments in the previous step also apply in this case yielding that for each $t>0$,
$\f_{t,j}$ converges to some $\psi_t$ in $\Cc^{\infty}(X)$ as $j$ goes to $+\infty$. Note that, by construction, $\psi_t$ is such that
$$
\frac{\partial \psi_t}{\partial t} = \log \left[\frac{(\theta_t+dd^c \psi_t)^n}{\omega^n}\right], \ \forall t>0.
$$
It follows from Lemma \ref{lem: lower bound on varphi t} that 
$$
\f_{t,j} \geq (1-t)\f_{0,j} - C(t), \forall t\in (0,1).
$$
Letting $j\to+\infty$, then $t\searrow 0$ we obtain 
$\liminf_{t\to 0} \p_t \geq \f_0$. To see the reverse inequality we set 
$$
u_{0,j}:= \left(\sup_{k\geq j}\f_{0,k}\right)^* \searrow \f_0.
$$
Let $u_{t,j}$ be the maximal solution of (\ref{eq: parabolic}) with initial data $u_{0,j}$. By the comparison principle we have 
$$
u_{t,j} \geq \f_{t,j}. 
$$
It follows from the previous step that 
 $(u_{t,j})$ decreases to $\f_t$, hence $\p_t\leq \f_t$. Since $\f_t$ converges in $L^1(X)$ to $\f_0$, we infer that $\p_t$ converges in $L^1(X)$ to $\f_0$. From the comparison principle (Theorem \ref{thm: comparison  zero}) we obtain $\f_t\equiv\p_t$.
\end{proof}

\subsection{Currents with positive Lelong numbers}
In this section we try and prove the uniqueness result for the general case: initial currents with positive Lelong numbers. We fix the following notations.
Assume that $\f_0$ is a $\omega$-psh function on $X$, $\eta$ is a smooth closed $(1,1)$-forms on $X$. Assume also that 
$$
\frac{1}{2c(\f_0)} < T_{\max}:=\sup\left\{t>0\setdef \{\omega\}+t\eta -t c_1(X)>0\right\}.   
$$
Fix two positive constants $\alpha, \beta$ such that 
$$
\frac{1}{2c(\f_0)} < \frac{1}{2\beta}<\frac{1}{2\beta-\alpha} < T_{\max}.
$$
Recall that the twisted K\"ahler-Ricci flow reads in the level of potentials as:
\begin{equation}
\label{eq: parabolic 2}
\dot{\f_t} = \log \left[\frac{(\theta_t +dd^c \f_t)^n }{\omega^n}\right],
\end{equation}
where $\theta_t:=\omega+t\eta-t\Ric(\omega)$. We recall the definition of weak solutions of the parabolic equation (\ref{eq: parabolic 2}).
\begin{defi}
\label{def: weak flow}
A function $\f:(0,T_{\max}) \times X \to \R$ is called a weak solution of the equation (\ref{eq: parabolic 2}) starting from $\f_0$ if the following conditions are satisfied:
\begin{itemize}
\item[(i)] For each $t>0$, the function $\f_t$ is $\theta_t$-psh on $X$. Moreover the function $t\mapsto \f_t$
is continuous as a map from $[0,T_{\max})$ to $L^1(X)$.
\item[(ii)]  For each $\vep>0$, there exists an analytic subset $D_{\vep}\subset X$ such that the function 
$$
(t,x)\mapsto \f_{t}(x)
$$
is smooth, strictly $\theta_t$-psh on $[\vep,T_{\max})\times (X\setminus D_{\vep})$ and satisfies the equation (\ref{eq: parabolic 2}) in the classical sense. Moreover, $D_{\vep}$ is empty if $\vep>1/(2c(\f_0))$. 
\end{itemize}
\end{defi}

We first show  that the maximal solution constructed in \cite{GZ13} is indeed  the maximal solution meaning that it dominates every weak solutions (in the sense of Definition \ref{def: weak flow}). 
\begin{lem}
\label{lem: actually maximal}
Let $\f_t$ be the maximal solution of (\ref{eq: parabolic 2}) starting from $\f_0$ and let $\p_t$ be a weak solution with initial data $\p_0$. Then 
$$
\p_0 \leq \f_0 \Longrightarrow \p_t \leq \f_t, \ \ \forall t\in (0,T_{\max}). 
$$
\end{lem}
\begin{proof}
Let $\f_{0,j}$ be the smooth sequence decreasing to $\f_0$ and $\f_{t,j}$ be the smooth sequence of approximants running from $\f_{0,j}$ and decreasing poitwise to $\f_t$. It thus suffices to show
that $\psi_t \leq \f_{t,j}$ for any fixed $j$. To simplify the notation, in the remaining part of the proof we drop the index $j$.  Fix $0<T<T_{\max}$ and $\varepsilon>\delta>0$. Denote by $Y_\epsilon$ the analytic set outside of which the weak solution $\psi_t$ is smooth for $t\geq \varepsilon$. Consider
$$H(t,x):= \psi_{t+\varepsilon}-\f_{t+\vep+\delta}+\delta\phi\qquad \; t\in[0,T],\; x\in X_\varepsilon$$
where $\phi\leq 0$ is a $\omega$-psh function that is smooth on $X_\varepsilon:= X\setminus Y_\varepsilon$ and goes to $-\infty$ on $Y_\varepsilon$. 
Observe that by construction $H$ is upper semicontinuos on $[0,T-\varepsilon]\times X_\varepsilon$ and $\phi=-\infty$ on $Y_\varepsilon$. This implies that $H$ achieves its maximum at some point $(t_0,x_0)\in [0,T-\varepsilon]\times X_\varepsilon$. We claim that $t_0=0$. Indeed, arguing by contradiction we get that at $(t_0,x_0)$
\begin{eqnarray*}
0\leq \frac{\partial}{\partial t}H &=& \dot{\psi}_{t+\varepsilon}-\dot{\f}_{t+\vep+\delta}\\
&=& \log\left[ \frac{(\theta_t+\varepsilon\chi +dd^c\psi_{t+\varepsilon} )^n}{(\theta_t +(\vep+\delta)\chi + dd^c\f_{t+\vep+\delta} )^n} \right]\\
&\leq &  \log\left[ \frac{(\theta_t +dd^c\f_{t+\vep+\delta} + \vep\chi - \delta dd^c \phi)^n}{(\theta_t +(\vep+\delta)\chi +dd^c\f_{t+\vep+\delta} )^n} \right] < 0
\end{eqnarray*}
where in the last inequality we assume $\chi\geq \omega$. Observe that this last assumption costs no generality since otherwise we can perform a change of variables and argue as in Section \ref{subsect: zero LeLong}. Thus we have, by letting $\delta\searrow 0$,
$$\p_{t+\vep}-\f_{t+\vep}\leq H(0,x_0)\leq \sup_X (\psi_{\varepsilon}-\f_{\vep}).$$
Furthermore, since $(t,x)\rightarrow \f_t(x)$ is smooth, and in particular continuous, it follows from Hartogs' lemma that
$$
\sup_X (\psi_{\varepsilon}-\f_{\vep})  \rightarrow \sup_X (\p_0-\f_{0})\leq 0.
$$
The conclusion follows when we let $\varepsilon$ go to zero.
\end{proof}

In the remaining part of this subsection we prove a "partial" uniqueness result in the case of positive Lelong numbers. 

\begin{lem}
\label{lem: initial inequ}
Assume that $\f_0$ is as above. Let $(\f_t)$ be a weak solution  of the equation (\ref{eq: parabolic 2}) starting from $\f_0$ and assume that $\f_t$ and  the maximal solution have the same singularities.
Then 
$$
(1-2\beta t)\f_0 - C(t) \leq  \f_t , \ \ \forall 0<t<(2\beta)^{-1},
$$
where $C(t)$ is a uniform constant depending only on $t$, and converges to zero as $t$ goes to zero. 
\end{lem}

\begin{proof}
We apply again the techniques in \cite[Lemma 2.9]{GZ13}. Fix two small positive constants $\vep, \delta$ such that 
\begin{equation}
\label{eq: ineq at zero 1}
(2\beta -\alpha -\delta)\omega +(1+2\beta \vep) \chi >0. 
\end{equation}
Consider the following function
$$
u_t= (1-2\beta t)\f_{\vep}+\alpha t v_{\vep} + n(t\log t -t), 
$$
where $v_{\vep}$ is the unique bounded function in $\psh(X,\omega)$ such that
\begin{equation}
\label{eq: ineq at zero 2}
\alpha^n (\omega +dd^c v_{\vep})^n = e^{\alpha v_{\vep}-2\beta\f_{\vep}}\omega^n. 
\end{equation}
Observe also that $v_{\vep}$ is uniformly bounded by a constant independent of $\vep$. 

Let $Y_{\vep}$ be the analytic subset outside of which the maximal solution is smooth for $t\geq \vep$.
Since $\f_t$ has the same singularities as the maximal solution, it ifollows that 
$$
\sup_{t\in [0,(2\beta)^{-1}]}\sup_{X_\vep} (u_t -\f_{t+\vep}) <+\infty.
$$
Here $X_{\vep}:=X\setminus Y_{\vep}$. Let $\phi$ be a negative $\omega$-plurisubharmonic function on $X$ which is smooth on $X_\vep$ and goes to $-\infty$ on $Y_{\vep}$.  By our construction the following function
$$
H(t,x) = u_t(x) - \f_{t+\vep}(x) + \delta t \phi , \ \ t\in [0,(2\beta)^{-1}], \ x\in X_{\vep},
$$
is upper semi-continuous on $[0,(2\beta)^{-1}]\times X_{\vep}$. Moreover $H(t,x)$ converges to $-\infty$ uniformly in $t$ as $x$ approaches $Y_{\vep}$. It then follows that $H$ attains a maximum at some $(t_0,x_0) \in [0,(2\beta)^{-1}]\times X_{\vep}$. By the choice of $\delta$ and $\vep$ and by the construction of $u_t$ it is easy to see that $t_0=0$. If it was not the case then by the maximum principle we would have, at this point,
\begin{equation}
\label{eq: ineq at zero 3}
0\leq \frac{\partial}{\partial t} H=\alpha v_{\vep}-2\beta\f_{\vep} +n\log t
 - \log \left[\frac{(\omega+ (t+\vep)\chi + dd^c \f_{t+\vep})^n}{\omega^n}\right].
\end{equation}
In the other hand, by (\ref{eq: ineq at zero 1}) and by the construction of  $u_t$ we have
\begin{equation*}
\alpha t(\omega + dd^c v_{\vep}) < \omega + (t+\vep) \chi + \delta t dd^c \phi  + dd^c u_t.
\end{equation*}
At the point $(t_0,x_0)$ where $H$ attains its maximum we have 
\begin{equation}
\label{eq: ineq at zero 4}
\alpha t(\omega + dd^c v_{\vep}) < \omega + (t+\vep) \chi + dd^c \f_{t+\vep}.
\end{equation}
Now, from (\ref{eq: ineq at zero 2}), (\ref{eq: ineq at zero 3}) and (\ref{eq: ineq at zero 4}) we get a contradiction.
Letting $\delta$ go to $0$ we obtain $u_t\leq \f_{t+\vep}$. We finally let $\vep\to 0$ to conclude. 
\end{proof}

We will prove the following uniqueness result.
\begin{thm}
\label{thm: uniqueness pos}
Assume that $\f_0$ is as above. Let $(\f_t)$ be a weak solution  of the equation (\ref{eq: parabolic 2}) starting from $\f_0$ and assume that $\f_t$ has the same singularities as the maximal solution.
Then $\f_t$ is also maximal.
\end{thm}
\begin{proof}
It costs no generality to assume that $\f_t\leq 0$ for small $t$. We first  treat the case when $\chi\geq 0$. 
Let $(\p_t)$ be the maximal solution of the equation (\ref{eq: parabolic 2}) starting from $\f_0$. We prove that $\f_t\equiv\p_t$. Since $\p_t$ is maximal, thanks to Lemma \ref{lem: actually maximal} it suffices to prove that $\p_t\leq \f_t$. 

Fix $T\in (0,T_{\max})$, $\vep>0$ and consider the following function
$$
u_t:= \f_{t+\vep} + C(\vep), \ t\in [0,T-\vep], \ x\in X_{\vep},
$$
where $C(\vep)$ and $X_{\vep}$ are defined as in Lemma \ref{lem: initial inequ}. We remark that by this construction and by Lemma \ref{lem: initial inequ} it holds that 
$$
u_0=\f_{\vep}+C(\vep)\geq \f_0. 
$$ 
Fix a small positive constant $\delta<\vep$ such that 
$\delta \omega <\vep \chi$. 
Consider now the following function 
$$
G(t,x):=\p_t(x) - u_t(x) + \delta \phi , \ \ t\in [0,T-\vep], \ x\in X_{\vep},
$$
where $\phi$ is also defined in Lemma \ref{lem: initial inequ}. Since $(\f_t)$ and $(\p_t)$ have the same singularity, meaning that the difference is uniformly bounded in $t$, one can apply the arguments in Lemma \ref{lem: initial inequ} to see that 
$G$ attains a maximum at some point $(t_0,x_0)$. Again, $t_0$ is forced to be $0$ by the maximum principle. Indeed, if $t_0$ were not $0$, by the maximum principle we would have, at $(t_0,x_0)$,
$$
0\leq \frac{\partial}{\partial t} G = \log \left[\frac{(\theta_t +dd^c \p_t)^n}{(\theta_t + \vep\chi +dd^c u_t)^n} \right] < \log \left[\frac{(\theta_t +\vep\chi + \delta dd^c \phi + dd^c \p_t )^n}{(\theta_t + \vep\chi +dd^c u_t)^n} \right] \leq 0,
$$
which is impossible. 

\medskip

\noindent{\bf Change of variables:} In the previous step we have assumed that $\chi\geq 0$. To remove this assumption we perform a change of variables exactly as in the case of zero Lelong numbers. 
It is harmless to assume that $\chi':=\omega +\chi > 0$. Now, let $(\f_t)$ be a weak solution and consider
$$
u_t:=e^{t}\f_{1-e^{-t}} , t\in [0,+\infty). 
$$
We compute 
\begin{equation}
\label{eq: KRF change variables}
\dot{u_t}= u_t -nt + \log \left[\frac{\omega +(e^t-1)\chi' +  dd^c u_t)^n}{\omega^n}\right].
\end{equation}
Thus, any weak (maximal) solution of the equation (\ref{eq: parabolic 2}) is in one-to-one correspondence with a weak (maximal) solution of the equation (\ref{eq: KRF change variables}) by a change of variables. 
Now, we combine the  previous step with the proof of Theorem \ref{thm: comparison zero} to conclude. 
\end{proof}

\subsection{Proof of Theorem B and Theorem B'}
Theorem B follows from Theorem \ref{thm: comparison zero} and Theorem \ref{thm: stability zero}. Theorem B'  follows from Theorem \ref{thm: uniqueness pos}.

\section{Convergence at time zero}
\label{sect: convergence at zero}
In this section we investigate the convergence at zero of weak solutions of the K\"ahler-Ricci flow. Let $\f_0$ be a $\omega$-plurisubharmonic function on $X$ and assume that 
$$
\frac{1}{2c(\f_0)}<T_{\max}.
$$ 
This condition insures that the maximal flow constructed in \cite{GZ13} is well-defined.
 
Let $(\f_t)$ be the maximal solution of the equation (\ref{eq: parabolic 2}) starting from $\f_0$.  If $\f_0$ is continuous  Song-Tian \cite{ST09} proved that $\f_t$ converges uniformly to $\f_0$. When $\f_0$ has finite energy, it was shown in \cite{GZ13} that $\f_t$ converges in energy while the convergence in capacity always holds. If $\f_0$ is continuous in some open subset $U\subset X$ we prove in the following that the convergence is locally uniform in $U$. This generalizes the uniform convergence of \cite{ST09} when $U=X$. 
\begin{thm}
\label{thm: uniform convergence at zero}
Assume that $\f_0$ is continuous  in an open subset $U$ of $X$. Then $\f_t$ converges in $L^{\infty}_{\rm loc}(U)$  to $\f_0$.  
\end{thm}

\begin{proof}
It costs no generality to assume that $\f_t\leq 0$ for small $t>0$. It follows from \cite[Lemma 2.9]{GZ13} that there exist a uniform constant $C$ and a small time $t_0$ such that 
$$
\f_t \geq \f_s + n(t-s) \left[\log (t-s) -C\right], \ \forall 0\leq s \leq t\leq t_0. 
$$
Set $\p_t:= \f_t + nCt - n t \log t$. Using the convexity of the function $x \log x$ we obtain
$$
\p_t \geq \p_s - nt\log 2.
$$
Now, set $u_t=\p_t+ nt\log 4$ and observe that 
$$
u_t\geq u_{t/2}.
$$
Thus for each $t>0$ the sequence $u_{t_02^{-j}}$ decreasingly converges to $u_0=\f_0$ as $j$ goes to $+\infty$. By Dini's theorem the convergence is locally uniform in $U$. 
\end{proof}

When the initial current has analytic singularities we obtain the following expected convergence in its regular locus:

\begin{thm}
\label{thm: cvg zero exp}
Assume that $e^{\gamma\f_0}$ is smooth on $X$ for some positive constant $\gamma$ and $\f_0$ is strictly $\omega$-psh in $X$. Then the maximal solution $\f_t$ of the equation (\ref{eq: parabolic 2}) converges in $\Cc^{\infty}_{\rm loc}(\Omega)$ to $\f_0$ as $t$ goes to $0$. Here,  $\Omega:=\{\f_0>-\infty\}$.  
\end{thm}
\begin{proof}
First observe that, since $e^{\gamma\f_0}$ is smooth on $X$, $\Omega=\{\f_0>-\infty\}$ is an open set. 
Thanks to Lemma \ref{lem: independent} the maximal solution does not depend on the choice of the approximating sequence. Thus we can take
$$
\f_{0,j}:=\frac{1}{\gamma}\log \left( e^{\gamma\f_0}+\frac{1}{j} \right)
$$ 
as the smooth sequence decreasing to $\f_0$. To simplify the notation we can assume that $\gamma=1$. 
For each $j$, let $\f_{t,j}$ be the smooth solution of (\ref{eq: parabolic 2}) starting from $\f_{0,j}$. 
 Denote by $\omega_t:=\theta_t+dd^c\f_{t,j}$. We want to prove that there exists a positive constant $C$ such that $$\frac{1}{C} \omega\leq \omega_t\leq C \omega $$ on each compact subset $K\subset \Omega$. The result will then follow from \cite[Theorem 3.2.17]{BEG13}. We start proving the right-hand side inequality. 
 Since, in $K$, $\omega_t$ and $\omega$ are both strictly positive forms, it suffices to prove that $\tr_\omega {(\omega_t)}\leq C$ in
  $K\subset \Omega$. Note that on each compact $K\subset \Omega$, $\f_0$ is a $\delta\omega $-psh function for some $\delta\in (0,1)$. Furthermore, $\f_0$ is smooth on $\Omega$ 
  and by construction we have $\f_{0,j}\geq \f_0 $. Moreover 
$$\Delta \f_{0,j} \leq B e^{-\f_0},$$
for some uniform constant $B>0$.
Indeed by a simple computation we get
$$
dd^c e^{\f_{0,j}}=e^{\f_{0,j}}dd^c\f_{0,j}+e^{\f_{0,j}} d\f_{0,j} \wedge d^c \f_{0,j} \leq B\omega
$$
where the last inequality follows from the fact that $ e^{\f_{0}}$ is smooth on $X$.
We now apply Theorem \ref{thm: C2 est more} to get the uniform estimate on $\tr_\omega (\omega_t)$. 
It remains to prove that $C^{-1} \omega\leq \omega_t $ on each $K\subset \Omega$. By construction and from the fact that $\f_0$ is strictly $\omega$-psh we get
$$
c \omega^n < (\omega+dd^c \f_{0,j})^n = e^{\dot{\f}_{0,j}} \omega^n
$$
where $c$ is a uniform positive constant. This easily implies $\dot{\f}_{0,j}> \log c$. Applying Proposition \ref{prop: est dot more}, we infer that there exists a uniform cosntant $C_1>0$ such that 
$$
\dot{\f}_{t,j}\geq C_1(\f_0+1).
$$ 
By construction 
$\omega_t^n= e^{\dot{\f}_{t,j}}\omega^n$ and so, thanks to the above estimate, we infer that there exists a uniform constant $C_2=C(K)$ such that $\omega_t^n \geq {C_2^{-1}} \omega^n$ on each $K \subset \Omega$. Let now $\lambda_1,\cdots, \lambda_n$ be the eigenvalues of the matrix associated to $\omega_t$ with respect to $\omega$. Then the latter inequality means that $\lambda_1\cdots\lambda_n \geq \frac{1}{C_2}$. This combined with the previous estimate (i.e. $\max_i \lambda_i \leq C$) gives that $\min_i \lambda_i \geq \frac{1}{C_3}$ for some uniform constant $C_3$. Hence the conclusion.
\end{proof}

The same arguments as above can also be applied in other contexts. Assume that  $0<f\in L^1(X)$ is smooth in the complement of some closed subset $D\subset X$. In \cite{DL14a} we say that $f$ satisfies the condition $\mathcal{H}_{f}$ if 
$$
f=e^{\p^+-\p^-}, \ \p^{\pm} \in \psh(X, C\omega), \ \p^-\in L^{\infty}_{\rm loc}(X\setminus D), 
$$
for some positive constant $C$. 
Assume also that $f$ satisfies the compatibility condition $\int_X f\omega^n=\int_X \omega^n$, the latter will be normalized to be $1$. 
Let $\f_0\in \Ec(X,\omega)$ be the unique normalized solution of the complex Monge-Amp\`ere equation
$$
(\omega+dd^c \f_0)^n = f\omega^n. 
$$
We are going to prove the following convergence result: 
\begin{thm}
\label{thm: convergence at zero smooth case}
Assume the above setting and let $(\f_t)$ be the unique weak solution of the equation (\ref{eq: parabolic 2}) with initial data $\f_0$. Then  $\f_t$ converges to $\f_0$ in $\Cc^{\infty}_{\rm loc}(X\setminus D)$. 
\end{thm}
\begin{proof}
We approximate $\p^{\pm}$ by using Demailly's technique (see \cite{Dem94}). Let $\p^{\pm}_{j}$ denote  these approximants. We then use Yau's work \cite{Yau} to find $\f_{0,j}\in \psh(X,\omega)\cap \Cc^{\infty}(X)$ such that 
$$
(\omega +dd^c \f_{0,j})^n =c_je^{\p^{+}_j-\p^{-}_j}\omega^n,
$$
where $c_j$ are normalization constants. It was proved in \cite{DL14a} that $c_j$ converges to $1$ as $j$ goes to $+\infty$ and the following estimates hold:
\begin{itemize}
\item For each $\vep>0$ there exists $C_1=C(\vep)>0$ such that
\begin{equation*}
\f_{0,j}\geq \vep \p^-_j  - C_1 .
\end{equation*}
\item There exists a uniform constant $C_2>0$ such that 
\begin{equation*}
\Delta_{\omega}(\f_{0,j}) \leq e^{-C_2\p^{-}_j}.
\end{equation*}

\end{itemize}

Now, for each $j$ let $(\f_{t,j})$ be the unique smooth solution of the parabolic equation (\ref{eq: parabolic 2}) with initial data $\f_{0,j}$. It follows from Theorem \ref{thm: stability zero} that $\f_{t,j}$ converges to $\f_t$ the unique weak solution of (\ref{eq: parabolic 2}).  Using Theorem \ref{thm: C2 est more} we get a uniform estimate for $\tr_{\omega}(\omega_{t,j})$ on each compact subset $K\Subset X\setminus D$. This together with Proposition \ref{prop: est dot more} and \cite[Theorem 3.2.17]{BEG13} yield locally uniform bounds on all derivatives of $\f_t$. This explains the convergence in $\Cc^{\infty}_{\rm loc}(X\setminus D)$ of $\f_t$ to $\f_0$ and the result follows. 
\end{proof}


\begin{thebibliography}{99}


\bibitem{BT82} E.~Bedford, B.~A.~Taylor: {\it A new capacity for plurisubharmonic functions}, Acta Math. 149 (1982), no. 1-2, 1--40.




\bibitem{Bern13} B.~Berndtsson: {\it The openness conjecture and complex Brunn-Minkowski inequalities}, arXiv:1405.0989, to appear in Proceeding of the Abel Symposium 2013.


\bibitem{BK07} Z.~B\l{}ocki, S.~Ko\l{}odziej: {\it On regularization of plurisubharmonic functions on manifolds},  Proc. Amer. Math. Soc.  {\bf 135}  (2007),  no. 7, 2089--2093.
 



\bibitem{BEG13} S.~Boucksom, P.~Eyssidieux, V.~Guedj: {\it An introduction to the K\"ahler-Ricci flow}, 1-6, Lecture Notes in Math., 2086, Springer, Cham, (2013). 





\bibitem{CGP13} F. Campana, H. Guenancia, M. P\u aun: {\it Metrics with cone singularities along normal crossing divisors and holomorphic tensor fields}, Annales Scientifiques de l'ENS {\bf 46}, fascicule 6 (2013), 879-916.


\bibitem{Cao85} H.D.Cao: {\it Deformation of K\"ahler metrics to K\"ahler-Einstein metrics on compact K\"ahler manifolds}, Invent. Math. {\bf 81} (1985), no. 2, 359-372.

\bibitem{CD07} X.~X.~Chen, W.Ding: {\it Ricci flow on surfaces with degenerate initial metrics}, J. Partial Differential Equations {\bf 20} (2007), no. 3, 193-202. 



\bibitem{CTZ11} X.~X.~Chen, G.Tian, Z. Zhang: {\it On the weak K\"ahler-Ricci flow}, Trans. Amer. Math. Soc. {\bf 363} (2011), no. 6, 2849-2863.
 
\bibitem{CSz12} T. C. Collins, G. Sz\'ekelyhidi: {\it The twisted K\"ahler-Ricci flow},  Preprint arXiv:1207.5441. 

\bibitem{Dem92} J. P. Demailly: {\it Regularization of closed positive currents and intersection theory}, Journal of Algebraic Geometry, 1(3), 361-409.

\bibitem{Dem94} J. P. Demailly: {\it Regularization of closed positive currents of type $(1,1)$ by the flow of a Chern connection}, Contributions to complex analysis and analytic geometry, 105-126, Aspects Math., E26, Vieweg, Braunschweig, 1994.

 
 \bibitem{Dem14} J.~P.~Demailly: {\it On the cohomology of pseudoeffective line bundles}, arXiv:1401.5432.

\bibitem{DK01} J.~P.~Demailly, J.~Koll\'ar: {\it Semi-continuity of complex singularity exponents and K\"ahler-Einstein metrics on Fano orbifolds}, Ann. Sci. \'Ecole Norm. Sup. (4) {\bf 34} (2001), no. 4, 525--556. 

\bibitem{DL14a} E.~Di Nezza, H.~C.~Lu: {\it Complex Monge-Amp\`ere equations on quasi-projective varieties}, Journal f\"ur die reine und angewandte Mathematik, (2014), DOI: 10.1515/crelle-2014-0090.

\bibitem{DL14b} E.~Di Nezza, H.~C.~Lu: {\it Generalized Monge-Amp\`ere capacities}, International Mathematics Research Notices (2014), DOI: 10.1093/imrn/rnu166.

 
\bibitem{EGZ09} P.~Eyssidieux, V.~Guedj, A.~Zeriahi: {\it Singular K\"ahler-Einstein metrics}, J. Amer. Math. Soc. {\bf 22} (2009), 607-639.  

  
\bibitem{FJ05} C. Favre, M. Jonsson: {\it Valuations and multiplier ideals}, J. Amer. Math. Soc. {\bf 18} (2005), no. 3, 655-684.


\bibitem{GuZh} Q. Guan, X. Zhou: {\it Strong openness conjecture and related problems for plurisubharmonic functions}, arXiv:1401.7158.

\bibitem{GZ05} V.~Guedj, A.~Zeriahi: {\it Intrinsic capacities on compact K{\"a}hler manifolds}, J. Geom. Anal.  {\bf 15}  (2005),  no. 4, 607-639.

\bibitem{GZ07} V.~Guedj, A.~Zeriahi: {\it The weighted Monge-Amp{\`e}re energy of quasiplurisubharmonic functions}, J. Funct. An.  {\bf 250} (2007), 442-482.
\bibitem{GZ13} V.~Guedj, A.~Zeriahi: {\it Regularizing properties of the twisted K\"ahler-Ricci flow},  arXiv:1306.4089, to appear in  Journal f\"ur die reine und angewandte Mathematik.

\bibitem{Ham82} R. S. Hamilton: {\it Three-manifolds with positive Ricci curvature}, J. Differ. Geom. 17(2), 255-306 (1982).

\bibitem{Hor90} L. H\"ormander: {\it An introduction to complex analysis in several variables}, Third edition. North-Holland Mathematical Library, 7. North-Holland Publishing Co., Amsterdam, 1990. xii+254 pp.


\bibitem{Kil79} C. O. Kiselman: {\it Densit\'e des fonctions plurisousharmoniques}, Bull. Soc. Math. France,  107 (1979), 295-304.
\bibitem{Kol03} S. Ko{\l}odziej, {\em The complex Monge-Amp\`ere equation on compact K\"{a}hler manifolds},  Indiana Univ. Math. J. 52 (2003), no. 3, 667-686.

\bibitem{Kol98} S.~Ko\l{}odziej: {\it The complex Monge-Amp{\`e}re equation}, Acta Math. {\bf 180} (1998), no. 1, 69--117.

\bibitem{Per1} G. Perelman: {\it The entropy formula for the Ricci flow and its geometric applications}, Preprint arXiv:math/0211159.
\bibitem{Per2} G. Perelman: {\it Ricci flow with surgery on three-manifolds}, Preprint arXiv:math/0303109.
\bibitem{Per3} G. Perelman: {\it Finite extinction time for the solutions to the Ricci flow on certain three-manifolds}, Preprint arXiv:math/0307245.

\bibitem{Hiep} H. H. Pham: {\it The weighted log canonical threshold},
arXiv:1401.4833.


\bibitem{Siu74} Y.~T.~Siu: {\it Analyticity of sets associated to Lelong numbers and the extension of closed positive
currents}, Invent. Math., 27:53-156, 1974.

\bibitem{Siu87} Y.~T.~Siu: {\it Lectures on Hermitian-Einstein metrics for stable bundles and K\"ahler-Einstein metrics}, DMV Seminar, 8. Birkh\"auser Verlag, Basel, 1987. 

\bibitem{Sko} H.~Skoda: {\it Sous-ensembles analytiques d'ordre fini ou infini dans  
$\mathbb{C}^{n}$}, Bull. Soc. Math. France  {\bf 100}  (1972), 353-408. 

\bibitem{ST09} J.~Song, G.~Tian: {\it The K\"ahler-Ricci flow through singularities}, Preprint (2009) arXiv:0909.4898. 





\bibitem{SW13b} J. Song, B. Weinkove: {\it Contracting exceptional divisors by the K\"ahler-Ricci flow I}, Duke Math. J. {\bf 162} (2013), no. 2, 367-415.


\bibitem{SW14} J. Song, B. Weinkove: {\it Contracting exceptional divisors by the K\"ahler-Ricci flow II}, Proc. Lond. Math. Soc. (3) 108 (2014), no. 6, 1529-1561.


\bibitem{SzTo} G.~Sz\'ekelyhidi, V.~Tosatti: {\it Regularity of weak solutions of a complex Monge-Amp\`ere equation}, Anal. PDE {\bf 4} (2011), no. 3, 369-378. 






\bibitem{TZha06} G.~Tian, Z. Zhang: {\it On the K\"ahler-Ricci flow on projective manifolds of general type}, Chinese Ann. Math. Ser. B {\bf 27} (2006), no. 2, 179-192.




\bibitem{Tsu} H.~Tsuji: {\it Existence and degeneration of K{\"a}hler-Einstein metrics on minimal algebraic varieties of general type},
 Math. Ann. {\bf  281}  (1988),  no. 1, 123--133.


\bibitem{Yau}  S.~T.~Yau: {\it On the Ricci curvature of a compact K{\"a}hler manifold and the complex Monge-Amp{\`e}re equation I}, Comm. Pure Appl. Math. {\bf 31} (1978), no. 3, 339-411.  

\bibitem{Zer01} A. Zeriahi: {\it Volume and capacity of sublevel sets of a Lelong class of plurisubharmonic functions}, Indiana Univ. Math. J., 50 (2001), 671-703.
\end{thebibliography}
\end{document}